\documentclass[11pt]{article}%
\usepackage{amssymb,amsmath,amsfonts,amsthm,array,bm,bbm,color}%
\usepackage{graphicx}
\providecommand{\U}[1]{\protect\rule{.1in}{.1in}}
\numberwithin{equation}{section}

\newtheorem{theorem}{Theorem}[section]
\newtheorem{lemma}[theorem]{Lemma}

\newtheorem{proposition}[theorem]{Proposition}
\newtheorem{remark}[theorem]{Remark}

\def\<{\langle}
\def\>{\rangle}
\def\d{{\rm d}}
\def\L{\mathcal{L}}
\def\div{{\rm div}}
\def\E{\mathbb{E}}
\def\N{\mathbb{N}}
\def\P{\mathbb{P}}
\def\R{\mathbb{R}}
\def\T{\mathbb{T}}
\def\Z{\mathbb{Z}}

\def\eps{\varepsilon}

\begin{document}

\title{Enhanced dissipation for stochastic Navier-Stokes equations\\ with transport noise}
\author{Dejun Luo \smallskip\\
{\small Key Laboratory of RCSDS, Academy of Mathematics and Systems Science,}\\
{\small Chinese Academy of Sciences, Beijing 100190, China} \\
{\small School of Mathematical Sciences, University of Chinese Academy of Sciences,} \\
{\small Beijing 100049, China} \\
{\small luodj@amss.ac.cn} }

\maketitle

\vspace{-20pt}

\begin{abstract}
The phenomenon of dissipation enhancement by transport noise is shown for stochastic 2D Navier-Stokes equations in velocity form. In the 3D case, suppression of blow-up is proved for stochastic Navier-Stokes equations in vorticity form; in particular, quantitative estimate allows us to choose the parameters of noise, uniformly in initial vorticity bounded in $L^2$-norm, so that global solutions exist with a large probability sufficiently close to 1.
\end{abstract}



\section{Introduction}

In  fluid dynamics and engineering, the advection-diffusion equation
  $$\partial_t\phi + u\cdot \nabla \phi = \nu\Delta \phi$$
is widely used to model the time evolution of passive scalar $\phi$ (e.g. the temperature or the distribution density of some solute) in a fluid flow $u$ with small molecular diffusion $\nu>0$; see e.g. \cite{Constantin, Zlatos10, LTD11, Seis13, YaoZlatos, ACM19, FengIyer, CZDE, WZZ20, Wei21, GalGub} and the references therein. The presence of a suitable flow $u$ sometimes greatly speeds up the convergence to equilibrium, similar to the behavior under a much stronger diffusion operator; in the case $\nu=0$, one is often interested in the mixing efficiency of the fluid flow. Incompressible flows with dissipation-enhancing property can be used to stabilize various systems and suppress possible blow-up of certain quantities \cite{BedHe, FFIT, Iyer, CZDFM}. There are also studies on the stabilizing effects of the Couette flow (or more general shear flows) for some fluid equations, such as the 2D Euler equation \cite{IJ20} and the 2D Boussinesq system \cite{DWZ21}, see \cite{BGM19} for a survey. In the stochastic setting, L. Arnold et al. \cite{ArnoldCW, Arnold} proved in the 1980s that suitable noises stabilize some finite dimensional linear ODE with a coefficient matrix of negative trace. In a series of recent papers \cite{BBPS21, BedroBlum21, BBPS18}, Bedrossian et al. have shown the almost sure exponential mixing and dissipation enhancement by random flows which are solutions to stochastic Navier-Stokes equations, leading to a proof of Batchelor's conjecture on the spectrum of passive scalar turbulence \cite{BBPS19}. Partly motivated by some ideas in \cite{DKK04, BedroBlum21}, stabilization and enhanced dissipation by Kraichnan noise were proved in \cite{GessYar}.

According to arguments by separation of scales \cite{FlaPap, FlaPap21}, stochastic 2D fluid equations driven by multiplicative transport noise in Stratonovich form are suitable models in fluid dynamics, see also \cite{Holm, CFH19} for variational considerations. Here, transport noise is assumed to be spatially divergence free, playing the role of incompressible flows mentioned above and modelling the turbulent actions of small-scale fluid components on larger ones. Indeed, such equations have already been studied by many authors, see for instances \cite{BCF, MikRoz04, MikRoz05, BFM, HLN19}.  Recently, Galeati \cite{Galeati} proved that stochastic linear transport equations on the torus $\T^d$ converge, under a certain scaling limit of the transport noise to high Fourier modes, to a deterministic parabolic equation. Later on, we have applied in \cite{FGL} this idea to the vorticity form of stochastic 2D Euler equations with transport noise, showing that they converge weakly to the deterministic 2D Navier-Stokes equations in vorticity form; similar results were proved in \cite{LuoSaal, Luo21} for other equations. Rewriting the equations in mild form via the heat semigroup, we have improved in \cite{FGL21b} the above weak limit results by establishing some explicit convergence rates; moreover, this approach enables us to obtain some properties on mixing and dissipation enhancement for linear equations driven by transport noise, see \cite[Section 1.3]{FGL21b} and also \cite{FGL21c, FlaLuongo22} for stochastic heat equations with transport noise in a bounded domain or a channel. As an example of noise on the torus with such properties, we can take the Kraichnan-type model in turbulence, cf. \cite[Remark 1.8]{FGL21b}. In these limit results, a remarkable feature is that larger intensity of transport noise gives rise to bigger viscosity coefficient in the limit equation. We have made use of this fact to show that transport noise suppresses potential blow-up of solutions to 3D Navier-Stokes equations in vorticity form \cite{FlaLuo21, FHLN21}; we refer to \cite{FGL21a, Luo23} for similar results on other equations. The large deviation principle and central limit theorem underlying the scaling limits have also been studied in \cite{GalLuo} for stochastic linear transport equations and 2D Euler equations; see \cite{LuoTang23, LuoWang23} for related results on other models. A survey of some of the results can be found in \cite{Flan21}. Finally, we mention that there were lots of studies on the ergodicity and exponential mixing of stochastic 2D Navier-Stokes equations with additive noise which may be very degenerate, see e.g. \cite{FlaMas95, EMS01, Kuk-Shiri01, BKL02, HM06}.

In this work, we first consider the velocity form of stochastic 2D Navier-Stokes equations driven by transport noise:
  \begin{equation}  \label{intro-S2DNSE}
  \left\{ \aligned
  \d u + u\cdot \nabla u\,\d t+ \d \nabla p &= \nu \Delta u\,\d t + \circ \d W_t \cdot\nabla u , \\
  \nabla \cdot u=0, \quad u|_{t=0} & =u_0,
  \endaligned \right.
  \end{equation}
where $u=(u^1, u^2)$ denotes the fluid velocity, $p$ is a scalar pressure term, $\nu>0$ is the viscosity coefficient; $\circ \d$ means the Stratonovich stochastic differential, and $W= W(t,x)$ is a space-time noise on $\T^2$, white in time, colored and divergence free in space, see \eqref{noise} below for explicit choice. In this case, it is well known that the equation admits a unique global solution (cf. \cite{MikRoz05}) and our purpose is to study the phenomenon of dissipation enhancement by transport noise. Secondly, we consider the vorticity form of stochastic 3D Navier-Stokes equations, still perturbed by transport noise:
  \begin{equation}\label{intro-S3DNSE}
  \d \xi + (u\cdot \nabla\xi - \xi\cdot \nabla u)\,\d t = \nu \Delta \xi\,\d t + \Pi(\circ\d W_t\circ \nabla \xi),
  \end{equation}
where the fluid velocity $u$ has now three components $(u^1, u^2, u^3)$ and $\xi =\nabla\times u$ is the corresponding vorticity field, while $W$ is a space-time noise on $\T^3$; $\Pi$ is the Helmholtz-Leray projection operator which makes the noise part become divergence free, see Section \ref{subs-3D-NSE-main-result} for more detailed discussions. We will try to establish quantitative estimate on the probability of (possible) blow-up; as a consequence, we can choose noise parameters (e.g. intensity and Fourier modes of noise) so that global solution exists with a high probability sufficiently close to 1, see Theorem \ref{thm-3D-SNSE} in Section \ref{subs-3D-NSE-main-result} for the precise statement.

We take the system \eqref{intro-S2DNSE} as an example to explain the main difficulty of our arguments. Compared to \cite{FGL, FGL21a, Luo21} which deal with scalar setting, we stress that the solutions to stochastic Navier-Stokes equations are divergence free vector fields, and thus we need to project the noise part using the Helmholtz projection $\Pi$, as in \cite{FlaLuo21}. In the 2D case, the transport noise takes the more precise form
  $$\circ \d W_t \cdot\nabla u= \sqrt{2 \kappa}\, \sum_{k\in \Z^2_0} \theta_k \sigma_k\cdot\nabla u\circ \d W^k_t,$$
where $\kappa>0$ is noise intensity, $\Z^2_0 = \Z^2\setminus \{0\}$ consists of nonzero lattice points, $\theta= \{\theta_k\}_k \in \ell^2(\Z^2_0)$ stands for the coefficients of noise, $\{\sigma_k\}_k$ are canonical divergence free complex-valued vector fields on $\T^2$, and $\{W^k\}_k$ are independent complex Brownian motions, see Section \ref{subs-notation} below for the precise definitions. Applying the operator $\Pi$ to the first equation in \eqref{intro-S2DNSE} and transforming it into It\^o form, the noise becomes
  $$\Pi(\circ \d W_t \cdot\nabla u)= \sqrt{2 \kappa}\, \sum_k \theta_k \Pi(\sigma_k\cdot\nabla u)\, \d W^k_t
  + 2\kappa \sum_k \theta_k^2\, \Pi\big[\sigma_k\cdot\nabla \Pi(\sigma_{-k}\cdot\nabla u) \big]\, \d t; $$
the second term on the right-hand side is the Stratonovich-It\^o corrector which will be denoted as $S_\theta^{(2)}(u)$, the superscript representing the dimension 2. Similarly to the scalar case, it is relatively easy to show that the martingale part vanishes in the weak sense by choosing a suitable sequence of noises $W^N(t,x)$ with weaker and weaker spatial correlation. However, the treatment of the above Stratonovich-It\^o corrector is much more difficult. In the scalar case \cite{FGL, FGL21a, Luo21} where $\Pi$ is not needed, some elementary computations show that the corrector reduces to a constant multiple of the Laplacian operator; on the contrary, in the current vector-field case, due to the presence of projection $\Pi$ which is a nonlocal operator, the corrector $S_\theta^{(2)}(u)$ has a much more complicated expression in Fourier expansions, see e.g. Lemma \ref{lem-append-exress} below.

In the 3D case, for a suitable sequence of noise coefficients $\{\theta^N\}_{N\geq 1}$ (see \eqref{theta-N-def} in the next subsection), it was proved in \cite[Theorem 5.1]{FlaLuo21} that $S_{\theta^N}^{(3)}(v)$ converges in $L^2$-norm to $\frac35\kappa \Delta v$, where $v$ is any divergence free smooth vector field. This is the key ingredient for showing suppression of vorticity blow-up by transport noise, see \cite[Theorem 1.6]{FlaLuo21}. However, Theorem 5.1 in \cite{FlaLuo21} is only a qualitative result without convergence rate; in order to prove dissipation enhancement and blow-up probability estimates, we need a quantitative version of the convergence result for $S_{\theta^N}^{(d)}(v)\ (d=2,3)$, which will be done in Theorem \ref{thm-Ito-corrector} below. To this end, we shall substantially improve the computations in \cite[Section 5]{FlaLuo21}, finding suitable Sobolev norms of convergence and certain estimates uniform in $N$; more precisely, we shall prove that there exists $C>0$ such that for any $\alpha\in [0,1]$, $s\in \R$ and $N\ge 1$,
  \begin{equation}\label{quantitative-estimate}
  \big\|S_{\theta^N}^{(d)}(v)- c_d \kappa \Delta v \big\|_{H^{s-2-\alpha}} \le C \kappa N^{-\alpha} \|v\|_{H^{s}}
  \end{equation}
holds for any  $v\in H^s(\T^d,\R^d)$, where $c_2= 1/4$ and $c_3= 3/5$, see Theorem \ref{thm-Ito-corrector} for exact statements. Once we have the quantitative convergence estimate, we can write the It\^o equations of \eqref{intro-S2DNSE} and \eqref{intro-S3DNSE} in mild form by using the semigroup $P_t = e^{(\nu+ c_d\kappa)t \Delta}$; the mild formulation involves the error terms $S_{\theta^N}^{(d)}(\cdot)- c_d \kappa \Delta $, as well as stochastic convolutions. Following some arguments in \cite{FGL21b} (in particular, Theorems 1.9 and 1.5 therein), and making use of basic heat kernel properties and maximal estimates on stochastic convolutions (cf. Theorems \ref{thm-stoch-convol} and \ref{thm-stoch-convol-2} below), we can show the property of dissipation enhancement for \eqref{intro-S2DNSE} and estimate of blow-up probability for \eqref{intro-S3DNSE}.

We briefly discuss potential applications of our method in future works, especially the quantitative estimate \eqref{quantitative-estimate}. As mentioned above, for incompressible fluid dynamics equations driven by Stratonovich transport noise, except the vorticity form in 2D case, we need to deal with solutions which are divergence free vector fields. In this case, it is customary to apply the Helmholtz-Leray operator $\Pi$ to get rid of the pressure term, but then the operator also appears on the noise part. Passing to It\^o equations, we get a Stratonovich-It\^o corrector $S_{\theta^N}^{(d)}(\cdot)$ and the estimate \eqref{quantitative-estimate} is needed to yield quantitative convergence rate. This approach works for many fluid dynamical models, see for instance the recent paper \cite{LuoTang23} dealing with stochastic inviscid Leray-$\alpha$ model with transport noise. It may also be applied to the regularized 3D Boussinesq equations with fractional Laplacian, cf. \cite{BesFer17}. Another interesting problem is to study fluid equations with the so-called nonlinear transport noise, namely, the intensity of noise depends on the solution. Recently, for the vorticity form of stochastic 2D Navier-Stokes equations with nonlinear transport noise, we have shown in \cite{FLL23} that the equations converge weakly, in a suitable scaling limit of the noise, to a Smagorinsky type model in large eddy simulation (cf. \cite{Berselli}). It would be nice to prove a similar result for the velocity equations, which involve a quantity like $\Pi(f(|\nabla u|) \circ \d W_t \cdot\nabla u)$, $f$ being some scalar function; in this case, the Stratonovich-It\^o corrector is even more complicated and is left to future studies.

We finish the short introduction with the organization of the paper. In the remainder of this section we introduce some notation for functional spaces and the space-time noise used in the paper. In Section \ref{sec-main-result} we describe the stochastic 2D/3D Navier-Stokes equations with transport noise and state our main results: Theorems \ref{thm-NSE} and \ref{thm-3D-SNSE}. Section \ref{sec-preparations} contains the preparations needed for proving the main results, then the proofs will be presented in Sections \ref{sec-proof-2D-NSE} and \ref{sec-proof-3D-NSE}, respectively. Finally, we prove in Appendix \ref{sec-appendix} the quantitative limit result of the Stratonovich correctors, which is the key ingredient in the proofs of main results.

\subsection{Functional setting and noise} \label{subs-notation}

We introduce some notations that will be used in the paper. Let $\T^d=\R^d/\Z^d$ be the $d$-dimensional torus, $d=2,3$; $\Z_0^d = \Z^d \setminus \{0\}$ is the set of nonzero integer points. For $s\in \R$ and $m\in \N$, we write $H^s(\T^d,\R^m) = W^{s,2}(\T^d,\R^m)$ for the usual Sobolev spaces of $\R^m$-valued functions; $H^0(\T^d,\R^m) = L^2(\T^d,\R^m)$. If $m=1$ we simply write $H^s(\T^d),\, L^2(\T^d)$, or $H^s, L^2$ when the dimension $d$ is clear. The brackets $\<\cdot, \cdot\>$ stand for the inner product in $L^2$ and the duality between elements in $H^s$ and $H^{-s}$; the norms in these spaces will be written as $\|\cdot \|_{L^2}$ and $\|\cdot \|_{H^s}$, respectively. Let $\{e_k \}_k$ be the usual complex basis of $L^2(\T^d, \mathbb C)$: for any $k\in \Z^d$, $e_k(x)= e^{2\pi {\rm i} k\cdot x}$, ${\rm i}$ being the imaginary unit. As the fluid equations considered below preserve the means of solutions, we assume in this paper that the function spaces consist of functions on $\T^d$ with zero average. The notation $a\lesssim b$ means that there exists a constant $C>0$ such that $a\leq Cb$; if we want to emphasize the dependence of $C$ on some parameters, e.g. $\alpha,p$, then we write $a\lesssim_{\alpha,p} b$.

We shall need the Helmholtz-Leray projection $\Pi : L^2(\T^d, \R^d) \to H$ where $H$ is the closed subspace of $L^2(\T^d, \R^d)$ consisting of divergence free vector fields. For a general vector field $X\in L^2(\T^d, \R^d)$, we have the formal expression
  $$\Pi X= X- \nabla \Delta^{-1} (\nabla \cdot X); $$
if $X= \sum_{k\in \Z^d_0} X_k e_k$ is the Fourier series, where $\{X_k \}_k \subset \mathbb C^d$, then
  $$\Pi X= \sum_{k\in \Z^d_0} \bigg(I_d- \frac{k\otimes k}{|k|^2}\bigg) X_k e_k, $$
where $I_d$ is the identity matrix of order $d$. The projection $\Pi$ will be the main source of technical difficulty in the paper as it is a nonlocal operator. For any $s\in \R$, we can extend $\Pi$ to $H^s(\T^d, \R^d)$ and it holds $\|\Pi X\|_{H^s} \leq \| X\|_{H^s}$ for any $X\in H^s(\T^d, \R^d)$.

Next, we introduce the space-time noise used in the paper to perturb the equations:
  \begin{equation}\label{noise}
  W(t,x)= \sqrt{C_{d} \kappa}\, \sum_{k\in \Z^d_0}\sum_{i=1}^{d-1} \theta_{k} \sigma_{k,i}(x) W^{k,i}_{t},
  \end{equation}
where $C_{d}=d/(d-1)$ is a normalizing constant, $\kappa>0$ is the noise intensity and $\theta\in\ell^{2} =\ell^{2}(\Z^d_0)$, the space of square summable sequences indexed by $\Z_0^d$. $\{W^{k,i}:k\in\mathbb{Z}^{d}_{0}, i=1,\ldots,d-1\}$ are standard complex Brownian motions defined on a filtered probability space $(\Omega, \mathcal F, (\mathcal F_t), \P)$, satisfying
  \begin{equation}\label{noise.1}
  \overline{W^{k,i}} = W^{-k,i}, \quad\big[W^{k,i},W^{l,j} \big]_{t}= 2t \delta_{k,-l} \delta_{i,j}.
  \end{equation}
$\{\sigma_{k,i}: k\in\mathbb{Z}^{d}_{0}, i=1,\ldots,d-1\}$ are divergence free vector fields on $\T^d$ defined as
  $$\sigma_{k,i}(x) = a_{k,i} e_{k}(x),$$
where $\{a_{k,i}\}_{k,i}$ is a subset of the unit sphere $\mathbb{S}^{d-1}$ such that: (i) $a_{k,i}=a_{-k,i}$ for all $k\in \mathbb{Z}^{d}_{0},\, i=1,\ldots,d-1$; (ii) for fixed $k$, $\{a_{k,i}\}_{i=1}^{d-1}$ is an ONB of $k^{\perp}=\{y\in\mathbb{R}^{d}:y\cdot k=0 \}$. In this way, $\{\sigma_{k,i}\}_{k,i}$ is a CONS of $H$. It holds that $\sigma_{k,i}\cdot \nabla e_k = \sigma_{k,i}\cdot \nabla e_{-k} \equiv 0$ for all $k\in \Z^d_0$ and $1\leq i\leq d-1$. If $d=2$, then we can explicitly define
  \begin{equation}\label{noise-2D}
  a_{k,1}= a_k = \begin{cases}
  \frac{k^\perp}{|k|}, & k\in \Z^2_+, \\
  -\frac{k^\perp}{|k|}, & k\in \Z^2_-,
  \end{cases}
  \end{equation}
where $k^\perp= (k_2,-k_1)$ is now a vector and $\Z^2_0 = \Z^2_+ \cup \Z^2_-$ is a partition of $\Z^2_0$ satisfying $\Z^2_+ = -\Z^2_-$. In the 3D case we shall write $\sum_{k\in \Z^3_0} \sum_{i=1}^{2}$ simply as $\sum_{k,i}$.

We shall always assume that
\begin{itemize}
\item $\theta \in \ell^2$ is symmetric, i.e. $\theta_k = \theta_l$ for all $k,l\in \Z^d_0$ satisfying $|k|=|l|$;
\item $\|\theta \|_{\ell^2} = \big(\sum_k \theta_k^2 \big)^{1/2} =1$.
\end{itemize}
The noise $W$ has the spatial covariance function (use \eqref{noise.1})
  \begin{equation}\label{covariance-funct}
  Q(x,y)= \E[W(1,x) \otimes W(1,y)]= 2C_d \kappa \sum_{k,i} \theta_k^2 (a_{k,i} \otimes a_{k,i}) e^{2\pi {\rm i} k\cdot (x-y)}, \quad x,y\in \T^d ;
  \end{equation}
Note that for any $k\in \Z^d_0$, $\{\frac{k}{|k|}, a_{k,1}, \ldots, a_{k,d-1}\}$ is an ONB of $\R^d$, hence
  $$\frac{k\otimes k}{|k|^2}+ \sum_{i=1}^{d-1} (a_{k,i} \otimes a_{k,i})  = I_d; $$
it is not difficult to show that (see \cite[Section 2]{FlaLuo21} for a proof in the 3D setting)
  \begin{equation}\label{covariance-funct-1}
  Q(x,x)= 2C_d \kappa \sum_{k,i} \theta_k^2 (a_{k,i} \otimes a_{k,i}) = 2\kappa I_d.
  \end{equation}
Moreover, choosing suitable $\theta \in \ell^2$ we obtain Kraichnan type noise, cf. \cite[Remark 1.8]{FGL21b}.

In the following, in order to deal with the Stratonovich-It\^o correctors (see \eqref{Ito-correction} and \eqref{Stra-Ito-corrector}) involving the Helmholtz projection $\Pi$, we shall take a special sequence of $\{\theta^N \}_{N} \subset \ell^2(\Z^d_0)$ defined as below:
  \begin{equation}\label{theta-N-def}
  \theta^N_k =\frac1{\Lambda_N} \frac1{|k|^\gamma} \textbf{1}_{\{N\leq |k|\leq 2N\}}, \quad k\in \Z^d_0,
  \end{equation}
where $\Lambda_N= \big(\sum_{N\leq |k|\leq 2N} \frac1{|k|^{2\gamma}} \big)^{1/2}$, $N\geq 1$. Approximating the sum by integrals, it is easy to show that
   $$\Lambda_N \sim \begin{cases}
   \sqrt{\frac{C_d}{d-2\gamma} (2^{d-2\gamma}-1)}\, N^{d/2-\gamma},& \quad 0<\gamma<d/2, \\
   \sqrt{C_d \log 2}, & \quad \gamma=d/2, \\
   \sqrt{\frac{C_d}{2\gamma -d} (1- 2^{d- 2\gamma})}\, N^{d/2-\gamma}, & \quad \gamma>d/2,
   \end{cases} $$
where $C_d$ is a dimensional constant; therefore, $\|\theta^N \|_{\ell^\infty} = \sup_k \theta^N_k \sim N^{-d/2}$. We see that the value of $\gamma$ plays no role in the rate of decay as $N\to\infty$, and thus in the sequel we will fix a $\gamma>0$ and omit the dependence of various constants on it.

\section{Main results} \label{sec-main-result}

This section contains the descriptions of our models and the main results obtained in the paper: in Section \ref{subsec-2DNSE} we deal with the stochastic 2D Navier-Stokes equations in velocity form and perturbed by multiplicative transport noise, then we consider in Section \ref{subs-3D-NSE-main-result} the stochastic 3D Navier-Stokes equations in vorticity form, still with transport noise.

\subsection{Stochastic 2D Navier-Stokes equations}\label{subsec-2DNSE}

In this section we consider the stochastic Navier-Stokes equations \eqref{intro-S2DNSE} on $\T^2$, perturbed by transport noise. Since we are in dimension 2, the noise has the simpler expression
  $$W(t,x)= \sqrt{2 \kappa}\, \sum_{k\in \Z^2_0} \theta_k \sigma_k(x) W^k_t, $$
where $\sigma_k(x)= a_k e_k(x)$ with $a_k$ defined in \eqref{noise-2D}. Substituting it into \eqref{intro-S2DNSE} we obtain
  \begin{equation}\label{NSE}
  \left\{ \aligned
  \d u +u\cdot \nabla u\,\d t+ \d \nabla p &= \nu \Delta u\,\d t + \sqrt{2\kappa} \sum_k \theta_k\, \sigma_k \cdot\nabla u \circ \d W^k_t, \\
  \nabla \cdot u=0, \quad u|_{t=0} & = u_0,
  \endaligned \right.
  \end{equation}
where the turbulent pressure takes the form $\d \nabla p= \nabla\tilde p\,\d t+  \sum_k \nabla p_k \circ \d W^k$. The equations \eqref{NSE} have been studied by many authors in the past, see e.g. \cite{BCF, MikRoz04, MikRoz05}. In particular, it was proved in \cite{MikRoz05} (replacing $\T^2$ by $\R^2$) that \eqref{NSE} admit a pathwise unique global solution, provided that the noise is regular enough.

\begin{remark}
If we define the vorticity $\xi = \nabla^\perp \cdot u$ and transform the first equation in \eqref{NSE} into vorticity form, then what we get is
  $$\d\xi + u\cdot\nabla \xi \,\d t = \nu \Delta \xi\,\d t + \sqrt{2\kappa} \sum_k \theta_k\, \big(\sigma_k \cdot\nabla \xi + r(\sigma_k,u) \big) \circ \d W^k_t, $$
with the remainder
  $$\aligned
  r(\sigma_k,u) &= \partial_2 \sigma_k^1 \partial_1 u^1 + \partial_2 \sigma_k^2 \partial_2 u^1 - \partial_1 \sigma_k^1 \partial_1 u^2 - \partial_1 \sigma_k^2 \partial_2 u^2 \\
  &= (\partial_2 \sigma_k^1 + \partial_1 \sigma_k^2 ) \partial_1 u^1 + (\partial_2 u^1+ \partial_1 u^2)  \partial_2 \sigma_k^2,
  \endaligned $$
where in the second step we have used the facts $\nabla\cdot u = \nabla\cdot \sigma_k =0$. The above equation is different from the vorticity form of stochastic 2D Navier-Stokes equations studied in \cite{BFM, FGL, BM19, CL19}.
\end{remark}

Recall the Helmholtz projection $\Pi: L^2(\T^2,\R^2) \to H$. As usual, define $b(u,v) = \Pi(u\cdot \nabla v),\, u,v\in H$; let $b(u)= b(u,u)$. For simplicity we still write $\Pi\Delta $ as $\Delta$. Applying the projection $\Pi$ to \eqref{NSE} yields
  \begin{equation}\label{NSE-Stra}
  \d u + b(u) \,\d t = \nu \Delta u\,\d t + \sqrt{2\kappa} \sum_k \theta_k\, \Pi(\sigma_k \cdot\nabla u) \circ \d W^k_t.
  \end{equation}
It has the It\^o form
  \begin{equation}\label{NSE-Ito}
  \d u + b(u) \,\d t = \big(\nu \Delta u + S_\theta^{(2)}(u) \big)\,\d t + \sqrt{2\kappa} \sum_k \theta_k\, \Pi(\sigma_k \cdot\nabla u) \, \d W^k_t,
  \end{equation}
where, by \eqref{noise.1}, the Stratonovich-It\^o correction term is defined as (the superscript $(2)$ stands for the dimension 2)
  \begin{equation}\label{Ito-correction}
  S_\theta^{(2)}(u)= 2\kappa \sum_k \theta_k^2\, \Pi\big[ \sigma_k \cdot\nabla \Pi(\sigma_{-k} \cdot\nabla u) \big],
  \end{equation}
which, like the Laplacian, is a symmetric operator with respect to the $L^2$-inner product of divergence free vector fields.

\begin{remark}
If there were no projection $\Pi$ on the right-hand side, then we would have $S_\theta^{(2)} (u)= \kappa \Delta u$ which can be easily proved by using \eqref{covariance-funct-1} with $d=2$, and thus the viscosity coefficient of \eqref{NSE-Ito} would be $\nu+\kappa$ if the martingale part vanishes in a certain limit. The presence of $\Pi$ makes it much more complicated to show enhanced dissipation. We shall follow some of the arguments in \cite[Section 5]{FlaLuo21} to overcome this difficulty.
\end{remark}

Let $V= H \cap H^1(\T^2,\R^2)$. For any $u_0 \in H$, the unique solution $u$ to \eqref{NSE-Ito} (or \eqref{NSE-Stra}) has trajectories in $C(0,\infty; H) \cap L^2(0,\infty; V)$ and satisfies
  \begin{equation}\label{NSE-energy-balance}
  \P \mbox{-a.s.}, \quad \|u_t \|_{L^2}^2 + 2\nu \int_s^t \|\nabla u_r\|_{L^2}^2\,\d r = \|u_s \|_{L^2}^2, \quad t>s \geq 0
  \end{equation}
(cf. \cite[Theorem 2.2]{MikRoz05}). Note that it holds independently of the noise parameters $\kappa>0$ and $\theta\in \ell^2$. Using the Poincar\'e inequality $4\pi^2 \|u\|_{L^2}^2 \leq \|\nabla u\|_{L^2}^2$, it is easy to know that $\|u_t \|_{L^2}$ decays exponentially fast:
  \begin{equation}\label{expon-decay}
  \P \mbox{-a.s.}, \quad \|u_t \|_{L^2} \leq e^{-4\pi^2 \nu t} \|u_0 \|_{L^2} \quad \forall\, t\geq 0.
  \end{equation}
However, the viscosity coefficient $\nu$ is usually very small, thus the above result only gives us a slow decay rate; what we want to prove is that transport noise enhances dissipation and leads to much faster decay. Unlike the evolution equations for scalars studied in \cite{FGL21b}, the main difficulty here comes from the projection operator $\Pi$ appearing in the noise part of \eqref{NSE-Stra}. We shall prove

\begin{theorem}[Enhanced dissipation] \label{thm-NSE}
For any $p\geq 1$, $\lambda>0$ and $L>0$, there exists $ \kappa>0$ and $\theta\in \ell^2$ with the following property: for every $u_0\in H$  with $\|u_0 \|_{L^2} \leq L$, there exists a random constant $C>0$ with finite $p$-th moment, such that for the solution $u$ of equation \eqref{NSE-Stra} with initial condition $u_0$, we have $\mathbb{P}$-a.s.
\[
\| u_t \|_{L^2} \leq C e^{-\lambda t} \|u_0 \|_{L^2}\quad \mbox{for all } t\geq0.
\]
\end{theorem}

See Section \ref{sec-proof-2D-NSE} for the proof and for the choices of parameters; in particular, we shall take $\theta=\theta^N$ defined in \eqref{theta-N-def} for $N$ big enough. Compared to usual results on dissipation enhancement for passive scalars, the above theorem is stated for bounded initial data. The reason is due to the nonlinear part $b(u)= \Pi(u\cdot \nabla u)$, which is quadratic and leads to higher order terms in the estimates. It is interesting to note that the nonlinearity does not cause any trouble in deriving the energy balance \eqref{NSE-energy-balance}, the latter implying exponential decay of $\|u_t \|_{L^2}$ with rate $4\pi^2 \nu$ for any initial data $u_0\in H$. Therefore, for initial data with large $L^2$-norm, we can wait some time $t_0\, (\approx \frac1\nu \log \frac{\|u_0 \|_{L^2}}{L})$ until  $\|u_{t_0} \|_{L^2} \leq L$ and then apply Theorem \ref{thm-NSE} to get fast decay. In other words, for fixed noise parameters $(\kappa, \theta)$, it takes some finite time before transport noise plays the role of producing enhanced dissipation.

\subsection{Stochastic 3D Navier-Stokes equations in vorticity form}\label{subs-3D-NSE-main-result}

Since the seminal work of Leray \cite{Leray} (see also the reviews \cite{Feff, BianFlan}), it is well known that the 3D incompressible Navier-Stokes equations on $\T^3$ (we consider unitary viscosity for simplicity)
  \begin{equation}\label{3D-NSE}
  \left\{ \aligned
  \partial_t u + u\cdot \nabla u + \nabla p&= \Delta u, \\
  \nabla \cdot u=0, \quad u|_{t=0} & = u_0
  \endaligned \right.
  \end{equation}
have global weak solutions for $L^2$-initial data, but their uniqueness is unknown; on the contrary, for more regular initial data in $H^1$, strong solutions exist uniquely, but their global existence in time remains open. An easy way of understanding the second difficulty is to rewrite the above system in vorticity form:
  \begin{equation}\label{3D-NSE-vort}
  \partial_t \xi + u\cdot \nabla\xi - \xi\cdot \nabla u = \Delta \xi,
  \end{equation}
where $\xi= \nabla\times u$ is the vorticity field. The relation can be reversed as $u=K\ast \xi$ with $K$ the Biot-Savart kernel on $\T^3$; note that $u\in H^1$ corresponds to $\xi\in L^2$. In the nonlinear part, not only is the vorticity transported by the fluid velocity, but it is also stretched by the deformation operator $\nabla u$; the latter is the potential source of blow-up for the vorticity. Indeed, we have the well known energy estimate:
  \begin{equation}\label{3D-NSE-energy-identity}
  \frac{\d}{\d t} \|\xi\|_{L^2}^2 \leq -\|\nabla \xi\|_{L^2}^2+ C\|\xi\|_{L^2}^6,
  \end{equation}
where $C>0$ is a constant arising from Sobolev embedding. It is not difficult to show that, combined with the Poincar\'e inequality $\|\nabla \xi\|_{L^2}^2 \geq 4\pi^2 \| \xi\|_{L^2}^2$, there is a small $r_0>0$ such that global solutions to \eqref{3D-NSE-vort} exist for initial data $\xi_0$ with $\|\xi_0\|_{L^2}\leq r_0$, while for general initial data we only have a (finite) estimate on the possible blow-up time.

In this part, we consider the stochastic 3D Navier-Stokes equations in vorticity form, with a multiplicative noise of transport type:
  \begin{equation}\label{3D-NSE-vort-noise}
  \partial_t \xi + u\cdot \nabla\xi - \xi\cdot \nabla u = \Delta \xi + \Pi(\dot W_t\circ \nabla \xi).
  \end{equation}
Here, $\circ$ hints that we shall use the Stratonovich stochastic differential. We point out that the Helmholtz projection $\Pi$ is necessary since the other terms are all divergence free (each one of $u\cdot \nabla\xi$ and $\xi\cdot \nabla u$ is not necessarily divergence free, but the difference turns out to be so), while $\dot W_t\circ \nabla \xi$ is in general not divergence free, though both the noise $\dot W_t$ and the vorticity $\xi$ have this property.  As in the last subsection, the operator $\Pi$ makes it more difficult to show suppression of blow-up by transport noise. Let us remark that a more natural noise for \eqref{3D-NSE-vort} is of transport-stretching type (cf. \cite{CFH19}), i.e. $\dot W_t\circ \nabla \xi- \xi\circ \nabla \dot W_t$, in which case the projection $\Pi$ is not needed. There are however some difficulties that we are unable to overcome for the time being, and hence we restrict ourselves to the stochastic 3D Navier-Stokes equations \eqref{3D-NSE-vort-noise} with projected transport noise.

Recall the noise defined in \eqref{noise}; inserting it into \eqref{3D-NSE-vort-noise} we obtain the stochastic 3D Navier-Stokes equations in Stratonovich form:
  \begin{equation}\label{SNSE-vort}
  \d \xi + \L_u \xi\,\d t = \Delta \xi\,\d t + \sqrt{\frac32 \kappa} \sum_{k\in \Z^3_0} \sum_{i=1}^2 \theta_k \Pi( \sigma_{k,i}\cdot \nabla \xi) \circ \d W^{k,i}_t,
  \end{equation}
where we have denoted for simplicity $\L_u \xi= u\cdot \nabla\xi - \xi\cdot \nabla u$, i.e. the Lie derivative of $\xi$ with respect to $u$. We shall write $\sum_{k\in \Z^3_0} \sum_{i=1}^2$ simply as $\sum_{k,i}$. In It\^o form, the above equation reads as
  \begin{equation}\label{SNSE-vort-Ito}
  \d \xi + \L_u \xi\,\d t = \big[\Delta \xi + S_\theta^{(3)}(\xi)\big] \,\d t  +\sqrt{\frac32 \kappa} \sum_{k,i} \theta_k \Pi(\sigma_{k,i}\cdot \nabla \xi) \, \d W^{k,i}_t ,
  \end{equation}
where the Stratonovich-It\^o correction term is now given by
  \begin{equation}\label{Stra-Ito-corrector}
  S_\theta^{(3)}(\xi)= \frac32 \kappa \sum_{k,i} \theta_k^2\, \Pi\big[ \sigma_{k,i}\cdot \nabla \Pi (\sigma_{-k,i} \cdot\nabla \xi) \big].
  \end{equation}
We remark that this equation is equivalent to the Stratonovich equation \eqref{SNSE-vort}, which enjoys similar energy estimate \eqref{3D-NSE-energy-identity} as the deterministic system \eqref{3D-NSE-vort}, because the noise part vanishes in energy type computations. Therefore, a priori, we know that \eqref{SNSE-vort-Ito} is only locally well posed for general data, and it admits a unique global solution for initial vorticity $\xi_0$ with $L^2$-norm less than $r_0$, where $r_0$ is the same value as in the deterministic case (see the end of the paragraph containing \eqref{3D-NSE-energy-identity}).

Recall that $H$ is the space of square integrable and divergence free vector fields on $\T^3$ with zero mean. Given an initial vorticity $\xi_0\in H$, denote by $\xi(t;\xi_0,\kappa, \theta)$ the pathwise unique local solution to \eqref{SNSE-vort-Ito} and $\tau=\tau(\xi_0,\kappa,\theta)$ the random maximal time of existence. Denote by $B_H(L)=\{\xi\in H: \|\xi \|_{L^2} \leq L\}$ and recall the noise coefficients $\theta^N$ given in \eqref{theta-N-def} for some fixed $\gamma>0$. Here is the main result of this part.

\begin{theorem}[Quantitative estimate on blow-up probability]\label{thm-3D-SNSE}
Fix some $\alpha\in (0,1/2)$ and big $L>0$. There exist constants $C=C(\alpha, L,r_0)>0$ and $K(\alpha,L,r_0)>0$ such that if $\kappa\geq K(\alpha,L ,r_0)$, then for any $\xi_0 \in B_H(L)$, it holds
  \begin{equation}\label{thm-3D-SNSE.1}
  \P \big(  \tau(\xi_0,\kappa,\theta^N) =+\infty \big) \geq 1- C \bigg(\frac{\kappa^{1-\alpha}}{N^{\alpha}} + \frac{\kappa^{9/20}}{N^{3/5}} + \frac{\kappa}{N^{2\alpha}} \bigg).
  \end{equation}
\end{theorem}

This result will be proved in Section \ref{sec-proof-3D-NSE}. The parameter $\alpha$ is due to the introduction of cut-off in the nonlinear part, see \eqref{3D-NSE-cut-off} below. The constants $C(\alpha, L,r_0)$ and $K(\alpha, L,r_0)$ can be computed explicitly though their expressions are a little complicated. The above estimate shows that, given large $L>0$ which measures the size of initial data, we can fix some noise intensity $\kappa$ such that, uniformly in $\xi_0\in B_H(L)$, the probability of existence of global solutions to \eqref{SNSE-vort} can be made arbitrarily close to 1 by choosing the noise with sufficiently high Fourier modes.

\section{Preparations} \label{sec-preparations}

In this section we make some preparations for proving Theorems \ref{thm-NSE} and  \ref{thm-3D-SNSE}. Section \ref{subs-correctors} contains the quantitative convergence results for the Stratonovich-It\^o correctors which are the key ingredients of proofs in Sections \ref{sec-proof-2D-NSE} and \ref{sec-proof-3D-NSE}. We provide in Section \ref{subs-basic-estim} some basic estimates on the heat semigroup, and prove in Section \ref{subs-stoch-convol} some estimates on the stochastic convolution.

\subsection{Convergence of the Stratonovich-It\^o correctors}\label{subs-correctors}

Recall the Stratonovich-It\^o correctors $S_\theta^{(d)}(\cdot)$ defined in \eqref{Ito-correction} and \eqref{Stra-Ito-corrector}, respectively in the 2D and 3D case. Due to the presence of the Helmholtz-Leray projection operator $\Pi$, it is difficult to estimate $S_\theta^{(d)}(\cdot)$ for general $\theta\in \ell^2(\Z^d_0),\, d=2,3$. Therefore, we consider the special sequence $\{\theta^N \}_N \subset \ell^2(\Z^d_0)$ given in \eqref{theta-N-def}. This choice is inspired by \cite[Theorem 5.1]{FlaLuo21} which deals with the 3D case, and it is proved that $S_{\theta^N}^{(3)}(v)$ converges in $L^2(\T^3,\R^3)$ to $\frac 35\kappa \Delta v$ for smooth vector fields $v$. Here we improve this result by establishing  quantitative convergence rates.

\begin{theorem}\label{thm-Ito-corrector}
Let $\theta^N$ be defined as in \eqref{theta-N-def}, $N\geq 1$. There exists a constant $C>0$, independent of $N\geq 1$, such that for any $s\in \R$ and $\alpha\in [0,1]$, and
\begin{itemize}
\item (2D case) for any divergence free field $v\in H^s(\T^2,\R^2)$, it holds
  \begin{equation}\label{thm-Ito-corrector.1}
  \bigg\| S_{\theta^N}^{(2)} (v)- \frac14 \kappa \Delta v \bigg\|_{H^{s-2-\alpha}} \leq C \frac{\kappa}{N^{\alpha}} \|v \|_{H^s};
  \end{equation}
\item (3D case) for any divergence free field $v\in H^s(\T^3,\R^3)$, it holds
  \begin{equation}\label{thm-Ito-corrector.2}
  \bigg\| S_{\theta^N}^{(3)} (v)- \frac35 \kappa \Delta v \bigg\|_{H^{s-2-\alpha}} \leq C \frac{\kappa}{N^{\alpha}} \|v \|_{H^s}.
  \end{equation}
\end{itemize}
\end{theorem}

To prove this theorem, we shall follow some computations in the proof of \cite[Theorem 5.1]{FlaLuo21}, but we need a few key improvements in order to find uniform constants independent of $N$. The proof is a little long and will be postponed to the appendix, where we prove \eqref{thm-Ito-corrector.1} in detail and briefly discuss the changes needed for proving \eqref{thm-Ito-corrector.2}.

\subsection{Basic properties of the heat semigroup} \label{subs-basic-estim}

First we state a well known property of the heat semigroup $\{e^{t \Delta}\}_{t\geq 0}$, cf. \cite{Pazy}.

\begin{lemma}\label{lem:heat-kernel-estim}
Let $u\in H^{\alpha}$, $\alpha\in\mathbb{R}$. Then for any $\rho\geq0$, it holds $\|e^{t\Delta} u \|_{H^{\alpha+\rho}} \leq C_{\rho}\, t^{-\rho/2} \|u\|_{H^{\alpha}}$ for some constant increasing in $\rho$.
\end{lemma}

The next lemma gives the regularizing properties of convolution with $\{e^{\delta t \Delta}\}_{t\geq 0}$, where $\delta>0$ is a fixed number. We refer to \cite[Lemma 2.3]{FGL21b} for a proof of the first result; the second one follows from similar computations.

\begin{lemma}\label{lem:heat-kernel}
For any $a<b $,  $\alpha\in\R$ and any $f\in L^{2}(a,b;H^{\alpha})$, it holds
   $$\bigg\|\int_a^t e^{\delta(t-s) \Delta} f_{s}\, \mathrm{d} s \bigg\|_{H^{\alpha+1}}^{2} \lesssim\frac1{\delta} \int_a^t \| f_{s}\|_{H^{\alpha}}^{2} \, \mathrm{d} s, \quad\forall\, t\in [a,b]. $$
Similarly, we have
  $$\int_a^b \bigg\|\int_a^t e^{\delta (t-s)\Delta} f_s \,\d s \bigg\|_{H^{\alpha+2}}^2 \,\d t \lesssim \frac1{\delta^2} \int_a^b \|f_s \|_{H^\alpha}^2\,\d s. $$
\end{lemma}

\subsection{Estimates on the stochastic convolution} \label{subs-stoch-convol}

Let $\delta>0$ and $P_t= e^{\delta t \Delta}, \, t\geq 0$ be the heat semigroup. Let $\{f_t \}_{t\geq 0}$ be a stochastic process of vector fields on $\T^d$ satisfying the energy inequality
  \begin{equation}\label{regularity}
  \|f_t \|_{L^2}^2 + \nu \int_s^t \|\nabla f_r\|_{L^2}^2\,\d r \leq \|f_s \|_{L^2}^2, \quad \forall\, t>s \geq 0.
  \end{equation}
In this part we follow the ideas of \cite[Lemma 5.3]{FGL21b} to deduce an estimate on the stochastic convolution
  $$Z_{s,t} = \int_s^t P_{t-r}\,(\d W_r\cdot\nabla f_r) = \sqrt{C_d \kappa} \sum_{k,i} \theta_k \int_s^t P_{t-r} \Pi(\sigma_{k,i} \cdot\nabla f_r)\,\d W^{k,i}_r, $$
where $\Pi$ is the Helmholtz projection; we write $Z_t$ for $Z_{0,t}$. In the proof below, we will apply the BDG inequality to martingales taking values in general Hilbert spaces, see e.g. \cite[Proposition 2.2.9]{PR07}.

\begin{theorem}\label{thm-stoch-convol}
For any $n\geq 0$, and any $0<\alpha<1 \leq \frac d2 <\beta\le \frac d2 +2$, it holds
  $$\int_{n}^{n+1} \E \|Z_{n,t} \|_{L^2}^2 \,\d t \lesssim_{\alpha, \beta} \kappa\, \nu^{-\frac{\beta}{\alpha+\beta}} \delta^{-\frac{\alpha(2\beta+d+4)} {4(\alpha+ \beta)}} \|\theta \|_{\ell^\infty}^{\frac{2\alpha}{\alpha+\beta}} \E\|f_n \|_{L^2}^2.$$
\end{theorem}

\begin{proof}
Given an $\alpha\in (0,1)$, by the definition of $Z_{s,t}$, we have
  $$\aligned
  \E \|Z_{n,t} \|_{H^\alpha}^2 &= C_d \kappa \, \E \bigg\|\sum_{k,i} \theta_k \int_n^t P_{t-r} \Pi (\sigma_{k,i}\cdot\nabla f_r)\,\d W^{k,i}_r \bigg\|_{H^\alpha}^2 \\
  &\lesssim \kappa\, \E \bigg[\sum_{k,i} \theta_k^2 \int_n^t \| P_{t-r} \Pi(\sigma_{k,i}\cdot\nabla f_r) \|_{H^\alpha}^2 \,\d r \bigg] \\
  &\lesssim \frac{\kappa}{\delta^\alpha} \, \E \bigg[\sum_{k,i} \theta_k^2 \int_n^t \frac1{(t-r)^\alpha} \| \sigma_{k,i} \cdot\nabla f_r \|_{L^2}^2 \,\d r \bigg],
  \endaligned $$
where the last step is due to Lemma \ref{lem:heat-kernel-estim} and the fact that $\|\Pi X\|_{H^a} \leq \|X\|_{H^a}$ for all $X\in H^a(\T^d,\R^d)$ and $a\in \R$. Using the simple inequality $\|\sigma_{k,i} \cdot\nabla f_r \|_{L^2} \leq \| \nabla f_r \|_{L^2}$, we obtain
  $$\aligned
  \int_{n}^{n+1} \E \|Z_{n,t} \|_{H^\alpha}^2 \,\d t &\lesssim \frac{\kappa}{\delta^\alpha} \|\theta \|_{\ell^2}^2 \int_{n}^{n+1} \E \int_n^t \frac1{(t-r)^\alpha} \| \nabla f_r \|_{L^2}^2 \,\d r \d t \\
  &= \frac{\kappa}{\delta^\alpha} \|\theta \|_{\ell^2}^2\, \E \int_{n}^{n+1} \| \nabla f_r \|_{L^2}^2 \int_r^{n+1} \frac1{(t-r)^\alpha} \, \d t\, \d r \\
  &\leq \frac{\kappa}{(1-\alpha) \delta^\alpha} \E \int_{n}^{n+1} \| \nabla f_r \|_{L^2}^2\, \d r,
  \endaligned $$
where the last step follows from $\|\theta \|_{\ell^2}=1$. Now by \eqref{regularity} we arrive at
  \begin{equation}\label{thm-stoch-convol.1}
  \int_{n}^{n+1} \E \|Z_{n,t} \|_{H^\alpha}^2 \,\d t \lesssim_\alpha \frac{\kappa}{\delta^\alpha \nu} \E \|f_n\|_{L^2}^2.
  \end{equation}

Next, for $\beta\in (d/2, 2+d/2]$, we define $\eps=(2\beta -d)/4$ which belongs to $ (0,1]$; then,
  $$\aligned
  \E \|Z_{n,t} \|_{H^{-\beta}}^2 &\leq C_d \kappa\, \E\bigg[\sum_{k,i} \theta_k^2 \int_n^t \| P_{t-r} \Pi(\sigma_{k,i} \cdot \nabla f_r)\|_{H^{-d/2-2\eps}}^2 \,\d r \bigg] \\
  &\lesssim \frac{\kappa}{\delta^{1-\eps}} \E\bigg[\sum_{k,i} \theta_k^2 \int_n^t \frac{\| \sigma_{k,i} \cdot \nabla f_r \|_{H^{-1-d/2-\eps}}^2}{(t-r)^{1-\eps}} \,\d r \bigg],
  \endaligned $$
where the last step follows from Lemma \ref{lem:heat-kernel-estim}. One has
  \begin{equation}\label{thm-stoch-convol.1.5}
  \E \|Z_{n,t} \|_{H^{-\beta}}^2 \lesssim \frac{\kappa}{\delta^{1-\eps}} \|\theta \|_{\ell^\infty}^2\, \E \int_n^t \frac{1}{(t-r)^{1-\eps}} \sum_{k,i} \| \sigma_{k,i} \cdot \nabla f_r \|_{H^{-1-d/2-\eps}}^2 \,\d r.
  \end{equation}
Since $\sigma_{k,i}$ is divergence free, it holds
  $$\| \sigma_{k,i} \cdot \nabla f_r \|_{H^{-1-d/2-\eps}} = \|\nabla\cdot (\sigma_{k,i}\otimes f_r)\|_{H^{-1-d/2-\eps}} \lesssim \|\sigma_{k,i}\otimes f_r \|_{H^{-d/2-\eps}} \lesssim \|e_k f_r \|_{H^{-d/2-\eps}}. $$
We have
  $$\aligned \sum_k \|e_k f_r \|_{H^{-d/2-\eps}}^2 &= \sum_k \sum_l \frac1{|l|^{d+2\eps}} |\<e_k f_r, e_{l}\>|^2 = \sum_l \frac1{|l|^{d+2\eps}} \sum_k |\<f_r, e_{l-k}\>|^2 \\
  &= \|f_r \|_{L^2}^2 \sum_l \frac1{|l|^{d+2\eps}} \lesssim \frac1\eps \|f_r \|_{L^2}^2,
  \endaligned $$
where the last step follows from the estimate
  $$\sum_l \frac1{|l|^{d+2\eps}} \leq \int_{\{x\in \R^d: |x|\geq 1/2\}} \frac{\d x}{|x|^{d+2\eps}} \lesssim_d \int_{1/2}^\infty \frac{\d s}{s^{1+2\eps}} \lesssim \frac1\eps. $$
As a result, for any $r\in [n, n+1]$,
  $$\sum_{k,i} \| \sigma_{k,i} \cdot \nabla f_r \|_{H^{-1-d/2-\eps}}^2 \lesssim \frac1\eps \|f_r \|_{L^2}^2 \leq \frac1\eps \|f_n \|_{L^2}^2, $$
where the last step is due to \eqref{regularity}. Substituting this estimate into \eqref{thm-stoch-convol.1.5} yields
  $$\aligned
  \E \|Z_{n,t} \|_{H^{-\beta}}^2 &\lesssim \frac{\kappa}{\delta^{1-\eps}} \|\theta \|_{\ell^\infty}^2\, \E \int_n^t \frac{\|f_n \|_{L^2}^2}{\eps (t-r)^{1-\eps}} \,\d r \lesssim \frac{\kappa}{\delta^{1-\eps} \eps^2} \|\theta \|_{\ell^\infty}^2\, \E\|f_n \|_{L^2}^2.
  \endaligned $$
Hence, recalling that $\eps=(2\beta -d)/4>0$,
  \begin{equation}\label{thm-stoch-convol.2}
  \int_n^{n+1} \E \|Z_{n,t} \|_{H^{-\beta}}^2\,\d t \lesssim_{d,\beta} \frac{\kappa}{\delta^{1+(d-2\beta)/4}} \|\theta \|_{\ell^\infty}^2\, \E\|f_n \|_{L^2}^2.
  \end{equation}

Finally, for $0<\alpha<1 \leq \frac d2 <\beta \le 2+ \frac d2$, by interpolation
  $$\|\phi \|_{L^2} \leq \|\phi \|_{H^\alpha}^{\beta/(\alpha+\beta)} \|\phi \|_{H^{-\beta}}^{\alpha/(\alpha+\beta)}, \quad \forall \phi \in H^{\alpha},$$
we have
  $$\aligned
  \int_{n}^{n+1} \E \|Z_{n,t} \|_{L^2}^2 \,\d t &\leq \int_{n}^{n+1} \E\Big( \|Z_{n,t} \|_{H^\alpha}^{2\beta/(\alpha+\beta)} \|Z_{n,t} \|_{H^{-\beta}}^{2\alpha/(\alpha+\beta)} \Big) \,\d t \\
  &\leq \bigg(\int_{n}^{n+1} \E \|Z_{n,t} \|_{H^\alpha}^{2} \,\d t \bigg)^{\frac{\beta}{\alpha+\beta}} \bigg(\int_{n}^{n+1} \E \|Z_{n,t} \|_{H^{-\beta}}^{2} \,\d t \bigg)^{\frac{\alpha}{\alpha+\beta}}
  \endaligned $$
by H\"older's inequality. Inserting the estimates \eqref{thm-stoch-convol.1} and \eqref{thm-stoch-convol.2} into the right-hand side of the above inequality, we obtain the desired result.
\end{proof}

The following estimate on stochastic convolution has been proved in \cite[Corollary 2.6]{FGL21b}; there it is assumed that $f$ is a process of real-valued functions, but the same computations still work in the case of vector fields, similarly as in the above result. Note that we only require $f$ to be bounded in $L^2$ almost surely.

\begin{theorem}\label{thm-stoch-convol-2}
Assume that $\P$-a.s. $\sup_{t\in [0,T]} \|f_t \|_{L^2} \leq R$ (rather than \eqref{regularity}). Then for any $\beta\in (0, d/2]$ and any $\eps\in (0,\beta]$, it holds
  $$\E \bigg[\sup_{t\in [0,T]} \|Z_{t} \|_{H^{-\beta}}^2 \bigg] \lesssim_{\eps, T} \kappa\delta^{\eps-1} \|\theta \|_{\ell^\infty}^{4(\beta-\eps)/d} R^2. $$
\end{theorem}

\section{Proof of Theorem \ref{thm-NSE}} \label{sec-proof-2D-NSE}

Recall the stochastic 2D Navier-Stokes equation \eqref{NSE-Ito} in It\^o form:
  \begin{equation}\label{2D-NSE-Ito}
  \d u + b(u) \,\d t = \big(\nu \Delta u + S_\theta^{(2)}(u) \big)\,\d t + \sqrt{2\kappa} \sum_k \theta_k\, \Pi(\sigma_k \cdot\nabla u) \, \d W^k_t,
  \end{equation}
where the Stratonovich correction term is defined as
  $$ S_\theta^{(2)} (v)= 2\kappa \sum_k \theta_k^2\, \Pi\big[ \sigma_k \cdot\nabla \Pi(\sigma_{-k} \cdot\nabla v) \big]. $$
We remark again that $S_\theta^{(2)} (v)$ reduces to $\kappa \Delta v$ if there were no $\Pi$ on the right-hand side. It is difficult to compute the corrector for general $\theta\in \ell^2$, and hence we consider the special sequence $\{\theta^N \}_N \subset \ell^2$ given in \eqref{theta-N-def}. Though the solution $u$ to the stochastic 2D Navier-Stokes equation \eqref{2D-NSE-Ito} also depend on $N\geq 1$, we omit it for ease of notation.

Inspired by the quantitative estimate \eqref{thm-Ito-corrector.1} in Theorem \ref{thm-Ito-corrector}, and replacing $\theta$ by $\theta^N$ in \eqref{2D-NSE-Ito}, we rewrite it as follows:
  $$\d u + b(u) \,\d t = \Big(\nu+ \frac14 \kappa\Big) \Delta u \,\d t + \Big(S_{\theta^N}^{(2)} (u) -\frac14 \kappa \Delta u\Big)\,\d t + \sqrt{2\kappa} \sum_k \theta^N_k\, \Pi(\sigma_k \cdot\nabla u) \, \d W^k_t .$$
Thanks to \eqref{thm-Ito-corrector.1} and the energy identity \eqref{NSE-energy-balance}, it is clear that $\big\| S_{\theta^N}^{(2)} (u) -\frac14 \kappa \Delta u\big\|_{L^2(0,T; H^{-1-\alpha})} \to 0$ as $N\to \infty$; therefore, this part can be seen as a small perturbation of the equation. In this section we denote by $P_t= e^{(\nu+\kappa/4)t \Delta},\, t\geq 0$. For any $t>s\geq 0$, the above equation has the mild form
  \begin{equation}\label{NSE-mild}
  u_t = P_{t-s} u_s - \int_s^t P_{t-r} b(u_r)\,\d r + \int_s^t P_{t-r}\Big(S_{\theta^N}^{(2)} (u_r) -\frac14 \kappa \Delta u_r\Big)\,\d r + Z_{s,t},
  \end{equation}
where $Z_{s,t}$ denotes the stochastic convolution
  \begin{equation}\label{stoch-convol-NSE}
  Z_{s,t} = \sqrt{2\kappa} \sum_k \theta^N_k\int_s^t P_{t-r} \Pi(\sigma_k \cdot\nabla u_r) \, \d W^k_r.
  \end{equation}

Using \eqref{thm-Ito-corrector.1} and the mild formulation \eqref{NSE-mild}, we can prove the following result which, differently from \cite[Lemma 5.3]{FGL21b}, does not give us a uniform estimate of $\delta$ since $\|u_0 \|_{L^2}^2$ appears; such term is due to the nonlinearity $b(u)$ in \eqref{NSE-mild}.

\begin{lemma}\label{lem-NSE}
There exists $\delta>0 $ such that, for any $n\geq 0$,
  $$\E \| u_{n+1} \|_{L^2}^2 \leq \delta\, \E \| u_n\|_{L^2}^2,$$
where, for some $0<\alpha <1 <\beta\le 3$,
  $$\delta \lesssim_{\alpha, \beta} \frac1\kappa + \frac1{\kappa^2 \nu} \|u_0 \|_{L^2}^2 + \frac1{N^2\nu} + \kappa^{\frac{4\beta- \alpha(2\beta+d)}{4(\alpha+\beta)}} \nu^{-\frac{\beta}{\alpha+\beta}} \Big(\frac1N \Big)^{\frac{2\alpha}{\alpha+\beta}}. $$
In particular, $\delta$ can be as small as we want by first taking big $\kappa>0$, then choosing $N \geq 1$ large enough.
\end{lemma}

\begin{proof}
The energy balance \eqref{NSE-energy-balance} implies that, $\P$-a.s. $\|u_t \|_{L^2}$ is decreasing in $t>0$; hence, by the mild formulation \eqref{NSE-mild},
  $$\|u_{n+1} \|_{L^2}^2 \leq \int_n^{n+1} \|u_t \|_{L^2}^2\,\d t \lesssim I_1+ I_2 + I_3 + I_4, $$
where the four terms are defined in an obvious way. We estimate them one by one.

First, we have
  $$I_1 = \int_n^{n+1} \|P_{t-n} u_n \|_{L^2}^2\,\d t \leq \int_n^{n+1} e^{-8\pi^2 t(\nu +\kappa/4)} \|u_n \|_{L^2}^2\,\d t \lesssim \frac1\kappa \|u_n \|_{L^2}^2. $$
Next, we deal with the second term $I_2$ which involves the nonlinearity $b(u)$: by the second estimate in Lemma \ref{lem:heat-kernel},
  $$I_2 = \int_n^{n+1} \bigg\|\int_n^t P_{t-r} b(u_r)\,\d r \bigg\|_{L^2}^2\,\d t \lesssim \frac1{\kappa^2} \int_n^{n+1} \| b(u_r) \|_{H^{-2}}^2 \,\d r. $$
For any divergence free vector field $\phi\in C^1(\T^2, \R^2)$, one has
  $$|\<b(u_r), \phi\>|= |\<u_r \cdot \nabla u_r, \phi\>|= |\<u_r, u_r \cdot \nabla \phi\>| \leq \|u_r \|_{L^4}^2 \|\nabla \phi\|_{L^2} \leq \|u_r \|_{H^{1/2}}^2 \|\phi \|_{H^1}, $$
where the last step is due to the Sobolev embedding $H^{1/2}(\T^2) \subset L^4(\T^2)$. We conclude that
  $$\| b(u_r) \|_{H^{-2}} \lesssim \| b(u_r) \|_{H^{-1}} \lesssim \|u_r \|_{H^{1/2}}^2 \leq \|u_r \|_{L^2} \|u_r \|_{H^1}; $$
therefore,
  $$I_2 \lesssim \frac1{\kappa^2} \int_n^{n+1} \|u_r \|_{L^2}^2 \|u_r \|_{H^1}^2 \,\d r \leq \frac1{\kappa^2} \|u_0 \|_{L^2}^2 \int_n^{n+1} \|u_r \|_{H^1}^2 \,\d r \leq \frac1{\kappa^2 \nu} \|u_0 \|_{L^2}^2 \|u_n \|_{L^2}^2, $$
where in the last step we have used the energy balance \eqref{NSE-energy-balance}.

Now we deal with the third term: using the second inequality in Lemma \ref{lem:heat-kernel},
  $$\aligned
  I_3 &= \int_n^{n+1} \bigg\|\int_n^t P_{t-r}\Big( S_\theta^{(2)} (u_r) -\frac14 \kappa \Delta u_r\Big)\,\d r \bigg\|_{L^2}^2\,\d t\\
  & \lesssim \frac1{\kappa^2} \int_n^{n+1} \bigg\| S_\theta^{(2)} (u_r) -\frac14 \kappa \Delta u_r \bigg\|_{H^{-2}}^2\,\d r .
  \endaligned $$
Applying the first assertion in Theorem \ref{thm-Ito-corrector} and \eqref{NSE-energy-balance} leads to
  $$\aligned
  I_3 &\lesssim \frac1{\kappa^2} \int_n^{n+1} C \frac{\kappa^2}{N^2} \|u_r \|_{H^1}^2 \,\d r\lesssim \frac1{N^2\nu} \|u_n \|_{L^2}^2 .
  \endaligned $$

Finally, the last term $I_4$ involving the stochastic convolution $Z_{n,t}$ can be treated by applying Theorem \ref{thm-stoch-convol} with $\delta= \nu+\kappa/4,\, \|\theta^N \|_{\ell^\infty} \sim 1/N$ (since $d=2$):
  $$ \E I_4 =\int_{n}^{n+1} \E \|Z_{n,t} \|_{L^2}^2 \,\d t \lesssim \kappa^{\frac{4\beta- \alpha(2\beta+d)}{4(\alpha+\beta)}} \nu^{-\frac{\beta}{\alpha+\beta}} \Big(\frac1N \Big)^{\frac{2\alpha}{\alpha+\beta}} \E\|u_n \|_{L^2}^2. $$
Summarizing the estimates on $I_i\, (i=1,2,3,4)$ we complete the proof.
\end{proof}

\begin{remark}
The estimate on $I_2$ can be improved as below. For any divergence free smooth field $\phi$ and any $q>2$, applying H\"older's inequality with exponents $\frac1q + \frac1q + \frac{q-2}q=1$ yields
  $$|\<b(u_r), \phi\>|= |\<u_r, u_r \cdot \nabla \phi\>| \leq \|u_r \|_{L^q}^2 \|\nabla \phi\|_{L^{q/(q-2)}} \lesssim \|u_r \|_{H^{(q-2)/q}}^2 \|\nabla \phi\|_{H^1}, $$
where we have used Sobolev embeddings in dimension 2. Let $\eta=q-2$; this implies
  $$\|b(u_r)\|_{H^{-2}} \lesssim \|u_r \|_{H^{\eta/(2+\eta)}}^2 \leq \|u_r \|_{L^2}^{4/(2+\eta)} \|u_r \|_{H^1}^{2\eta/(2+\eta)}  $$
by interpolation. As a result,
  $$I_2 \lesssim \frac1{\kappa^2} \int_n^{n+1} \|u_r \|_{L^2}^{8/(2+\eta)} \|u_r \|_{H^1}^{4\eta/(2+\eta)} \,\d r \leq \frac1{\kappa^2} \|u_n \|_{L^2}^{8/(2+\eta)} \int_n^{n+1} \|u_r \|_{H^1}^{4\eta/(2+\eta)} \,\d r, $$
where the last step is due to the non-increasing property of $r\mapsto \|u_r \|_{L^2}$. For $\eta<2$, by Jensen's inequality and the energy balance \eqref{NSE-energy-balance},
  $$\aligned
  I_2 &\lesssim \frac{1}{\kappa^2} \|u_n \|_{L^2}^{8/(2+\eta)} \bigg(\int_n^{n+1} \|u_r \|_{H^1}^2 \,\d r \bigg)^{2\eta/(2+\eta)} \\
  & \leq \frac{1}{\kappa^2} \|u_n \|_{L^2}^{8/(2+\eta)} \bigg(\frac1\nu \|u_n \|_{L^2}^2 \bigg)^{2\eta/(2+\eta)} \\
  & = \frac{1}{\kappa^2 \nu^{2\eta/(2+\eta)}} \|u_n \|_{L^2}^4 \leq \frac{\|u_0 \|_{L^2}^2}{\kappa^2 \nu^{2\eta/(2+\eta)}}  \|u_n \|_{L^2}^2.
  \endaligned $$
When the viscosity coefficient $\nu$ is small,  this estimate improves the above one if we take $\eta>0$ very small.
\end{remark}

With the above estimates in mind, we can follow the ideas of the proof of \cite[Theorem 1.]{FGL21b} to show Theorem \ref{thm-NSE}.

\begin{proof}[Proof of Theorem \ref{thm-NSE}]
Lemma \ref{lem-NSE} implies that there exists a small $\delta \in (0,1)$ such that for any $n\geq 1$,
  $$\E \|u_n \|_{L^2}^2 \leq \delta\, \E \|u_{n-1} \|_{L^2}^2 \leq \cdots \leq \delta^n \|u_0 \|_{L^2}^2. $$
Recalling that $t\to \|u_t \|_{L^2}$ is $\P$-a.s. decreasing, we have
  $$\E \bigg(\sup_{t\in [n,n+1]} \|u_t \|_{L^2}^2 \bigg) = \E \|u_n \|_{L^2}^2 \leq \delta^n \|u_0 \|_{L^2}^2 = e^{-2\lambda' n} \|u_0 \|_{L^2}^2, $$
where $\lambda'= -\frac12 \log \delta >0$. By Lemma \ref{lem-NSE}, we can choose a suitable pair $(\kappa,\theta^N)$ such that $\lambda'> \lambda(1+p/2)$, where $\lambda>0$ and $p\geq 1$ are parameters in the statement of Theorem \ref{thm-NSE}.

Now for any $n\geq 1$, we define
  $$A_n = \bigg\{\omega\in \Omega: \sup_{t\in [n,n+1]} \|u_t(\omega) \|_{L^2} > e^{-\lambda n} \|u_0 \|_{L^2} \bigg\}.$$
Then by Chebyshev's inequality,
  $$\sum_n \P(A_n) \leq \sum_n \frac{e^{2\lambda n}}{\|u_0 \|_{L^2}^2} \E\bigg(\sup_{t\in [n,n+1]} \|u_t \|_{L^2}^2 \bigg) \leq \sum_n e^{2(\lambda -\lambda')n} <+\infty, $$
therefore, by Borel-Cantelli lemma, for $\P$-a.s. $\omega\in \Omega$, there exists a big $N(\omega) \geq 1$ such that
  $$\sup_{t\in [n,n+1]} \|u_t(\omega) \|_{L^2} \leq e^{-\lambda n} \|u_0 \|_{L^2} \quad \forall\, n> N(\omega).$$
For $0\leq n\leq N(\omega)$, we have
  $$\sup_{t\in [n,n+1]} \|u_t(\omega) \|_{L^2} \leq \|u_n(\omega) \|_{L^2} = e^{\lambda n} e^{-\lambda n} \|u_n(\omega) \|_{L^2} \leq e^{\lambda N(\omega)} e^{-\lambda n} \|u_0 \|_{L^2}. $$
Thus, if we take $C(\omega)= e^{\lambda (1+N(\omega))} $, then it is easy to show that, $\P$-a.s. for all $t\geq 0$, $\|u_t(\omega) \|_{L^2} \leq C(\omega) e^{-\lambda t} \|u_0 \|_{L^2}$.

It remains to estimate the $p$-th moment of the random variable $C(\omega)$; to this end, we need to estimate the tail probability $\P(\{N(\omega) \geq k\})$. Note that $N(\omega)$ may be defined as the largest integer $n$ such that $\sup_{t\in [n,n+1]} \|u_t(\omega) \|_{L^2} > e^{-\lambda n} \|u_0 \|_{L^2}$; hence
  $$\{\omega\in \Omega: N(\omega) \geq k\} = \bigcup_{n=k}^\infty A_n. $$
Then, we have
  $$\P(\{N(\omega) \geq k\}) \leq \sum_{n=k}^\infty \P(A_n) \leq \sum_{n=k}^\infty e^{2(\lambda -\lambda')n} = \frac{e^{2(\lambda -\lambda')k}}{1- e^{2(\lambda -\lambda')}}. $$
As a result,
  $$\E e^{\lambda p N(\omega)} = \sum_{n=0}^\infty e^{\lambda p k} \P(\{N(\omega) = k\}) \leq \frac1{1- e^{2(\lambda -\lambda')}} \sum_{n=0}^\infty e^{\lambda p k} e^{2(\lambda -\lambda')k} < +\infty ,$$
where the last step is due to the choice of $\lambda'$. Therefore $C(\omega)$ has finite $p$-th moment.
\end{proof}

\section{Proof of Theorem \ref{thm-3D-SNSE}} \label{sec-proof-3D-NSE}

This section consists of two parts: we provide in Section \ref{subs-3D-NSE-preliminary} some preliminary estimates on the limit equation \eqref{3D-NSE-limit} and on the approximating equations \eqref{3D-NSE-cut-off} with cut-off, together with stability estimates on the nonlinearity (see Propositions \ref{prop-nonlin-3D} and \ref{prop-nonlin-3D-1}); the proof of Theorem \ref{thm-3D-SNSE} is given in Section \ref{subs-3D-NSE-proof}, where we estimate the blow-up probability by using the mild formulations of both approximating equations and limit equation, as well as the key ingredient \eqref{thm-Ito-corrector.2}.

\subsection{Some preliminary results}\label{subs-3D-NSE-preliminary}

Recall the stochastic 3D Navier-Stokes equations in vorticity form:
  \begin{equation}\label{SNSE-vort-Ito-1}
  \d \xi + \L_u \xi\,\d t = \big[\Delta \xi + S_\theta^{(3)}(\xi)\big] \,\d t  +\sqrt{\frac32 \kappa} \sum_{k,i} \theta_k \Pi(\sigma_{k,i}\cdot \nabla \xi) \, \d W^{k,i}_t ,
  \end{equation}
where the Stratonovich-It\^o correction term is
  \begin{equation}\label{Stra-Ito-corrector-3D}
  S_\theta^{(3)}(\xi)= \frac32 \kappa \sum_{k,i} \theta_k^2\, \Pi\big[ \sigma_{k,i}\cdot \nabla \Pi (\sigma_{-k,i} \cdot\nabla \xi) \big].
  \end{equation}
Thanks to the second claim in Theorem \ref{thm-Ito-corrector}, taking $\theta=\theta^N$ in \eqref{SNSE-vort-Ito-1} and letting $N\to \infty$, if we can show that the martingale part vanishes, then the limit equation would be the vorticity form of deterministic 3D Navier-Stokes equations
  \begin{equation}\label{3D-NSE-limit}
  \partial_t \xi + \L_u \xi = \Big(1+ \frac35 \kappa \Big) \Delta \xi.
  \end{equation}
We first prove the following basic estimate on the solutions to \eqref{3D-NSE-limit}; in this section $H \subset L^2(\T^3,\R^3)$ is the subspace of divergence free vector fields on $\T^3$.

\begin{lemma}\label{lem-3D-NSE-limit}
There exists $C_1>0$ with the following property: given $\xi_0\in H$, if $\kappa\geq C_1 \|\xi_0 \|_{L^2}^4 +1$, then  \eqref{3D-NSE-limit} admits a unique global solution $\{\xi_t \}_{t\geq 0}$ satisfying
  \begin{equation}\label{lem-3D-NSE-limit.1}
  \max \bigg\{\sup_{t\geq 0} \big( e^{2t} \|\xi_t \|_{L^2}^2 \big) ,\, \int_0^\infty \|\nabla\xi_t \|_{L^2}^2\,\d t \bigg\} \leq \|\xi_0 \|_{L^2}^2.
  \end{equation}
\end{lemma}

\begin{proof}
The estimates below are close in spirit to those in \cite[Lemma 4.3]{FGL21b}. Let $\kappa_1= 1+ \frac35 \kappa$; using \eqref{3D-NSE-limit} it holds
  $$\frac{\d}{\d t} \|\xi \|_{L^2}^2 = 2\<\xi, -\L_u\xi + \kappa_1 \Delta\xi \> = 2\<\xi, \xi\cdot \nabla u\>- 2\kappa_1 \|\nabla \xi\|_{L^2}^2, $$
where we have used $\<\xi, u\cdot\nabla\xi \>=0$ since $u$ is divergence free. By H\"older's inequality and Sobolev embedding in 3D, we obtain
  $$|\<\xi, \xi\cdot \nabla u\>| \leq \|\xi \|_{L^3}^2 \|\nabla u\|_{L^3} \lesssim \|\xi \|_{L^3}^3 \lesssim \|\xi \|_{H^{1/2}}^3 \lesssim \|\xi \|_{L^2}^{3/2} \|\nabla \xi \|_{L^2}^{3/2}, $$
where the last step is due to interpolation of Sobolev norms. Applying Young's inequality with exponents $\frac14 + \frac34 =1$ leads to
  $$|\<\xi, \xi\cdot \nabla u\>| \leq c_0 \|\xi \|_{L^2}^6 + \frac12 \|\nabla \xi \|_{L^2}^2. $$
Hence, it holds
  $$\frac{\d}{\d t} \|\xi \|_{L^2}^2 + \|\nabla \xi\|_{L^2}^2 \leq 2c_0 \|\xi \|_{L^2}^6 - \frac65 \kappa \|\nabla \xi\|_{L^2}^2 .$$
Recall the Poincar\'e inequality $4\pi^2 \|\xi \|_{L^2}^2 \leq \|\nabla \xi\|_{L^2}^2$; we have
  \begin{equation}\label{lem-3D-NSE-limit.2}
  \aligned
  \frac{\d}{\d t} \|\xi \|_{L^2}^2 +\|\nabla \xi\|_{L^2}^2 &\leq 2c_0 \|\xi \|_{L^2}^6 -\frac{24\pi^2}5 \kappa \| \xi\|_{L^2}^2 \\
  &\leq -\Big(\frac{24\pi^2}5 \kappa - 2c_0 \|\xi \|_{L^2}^4 \Big) \|\xi \|_{L^2}^2 .
  \endaligned
  \end{equation}

Now we claim that the constant
  $$C_1= \frac5{12\pi^2} c_0 $$
fulfills our requirement. Indeed, if $\kappa\geq C_1 \|\xi_0 \|_{L^2}^4 +1$, then
  $$\frac{24\pi^2}5 \kappa - 2c_0 \|\xi_0 \|_{L^2}^4 \geq \frac{24\pi^2}5 \Big(\frac5{12\pi^2} c_0 \|\xi_0 \|_{L^2}^4 +1\Big) - 2c_0 \|\xi_0 \|_{L^2}^4 = \frac{24\pi^2}5 >1; $$
combining this with \eqref{lem-3D-NSE-limit.2} we deduce that $\frac{\d}{\d t} \|\xi_t \|_{L^2}^2<0 $ for $t=0$, thus the energy  $\|\xi _t\|_{L^2}^2$ is decreasing for small time, and also $\frac{24\pi^2}5 \kappa - 2c_0 \|\xi_t \|_{L^2}^4 >1$ will be true for all later times. As a result, we have
  $$\frac{\d}{\d t} \|\xi \|_{L^2}^2 \leq - \|\xi \|_{L^2}^2$$
implying $\|\xi_t \|_{L^2} \leq e^{-t} \|\xi_0 \|_{L^2}$ for all $t\geq 0$. We also deduce from \eqref{lem-3D-NSE-limit.2} that
  $$\int_0^\infty \|\nabla \xi_t \|_{L^2}^2\,\d t \leq \|\xi_0 \|_{L^2}^2. $$
We complete the proof.
\end{proof}

Due to the nonlinear term in \eqref{SNSE-vort-Ito-1}, solutions might blow up in finite time; thus we introduce a cut-off in this part. Let $\alpha\in (0,1/2)$; we consider the following stochastic 3D Navier-Stokes equations with cut-off:
  \begin{equation}\label{3D-NSE-cut-off}
  \aligned
  \d\xi + g_{\alpha, R}(\xi) \L_u\xi \,\d t &= \Delta\xi\,\d t + \sqrt{\frac32 \kappa} \sum_{k,i} \theta_k \Pi( \sigma_{k,i} \cdot\nabla \xi)\circ \d W^{k,i} \\
  &= \big[\Delta\xi +S_\theta^{(3)}(\xi) \big]\,\d t + \sqrt{\frac32 \kappa} \sum_{k,i} \theta_k \Pi( \sigma_{k,i} \cdot\nabla \xi)\, \d W^{k,i},
  \endaligned
  \end{equation}
where $g_{\alpha, R}(\xi)= g_R(\|\xi \|_{H^{-\alpha}})$ is a cut-off function, with $g_R\in C([0,\infty); [0,1])$ satisfying $g_R\equiv 1$ on $[0,R]$ and $g_R\equiv 0$ on $[R+1,\infty)$, having Lipschitz constant 1. Note that if $\xi$ solves \eqref{3D-NSE-cut-off} on some interval $[0,T]$ and $\sup_{t\in [0,T]} \|\xi_t \|_{H^{-\alpha}} \leq R$, then it also solves the original stochastic 3D Navier-Stokes equations \eqref{SNSE-vort-Ito-1} without cut-off. We would like to stress that the cut-off is carefully designed according to the nonlinearities of partial differential equations (PDEs); here, for $\alpha$ restricted to $(0,1/2)$, the cut-off based on $H^{-\alpha}$-norm of solution works for our purpose, also in the deterministic case (i.e. the first equation in \eqref{3D-NSE-cut-off} without noise). For other nonlinear PDEs we need different choices of cut-off. Thanks to the cut-off, global existence of solutions holds for \eqref{3D-NSE-cut-off} for any initial data $\xi_0 \in H$; this can be proved by first showing the existence of weak solutions and pathwise uniqueness, then applying the Yamada-Watanabe theorem, see e.g. \cite[Theorem 1.3]{FlaLuo21} which makes use of an estimate similar to Lemma \ref{lem-3D-NSE-cut-off} below. From this estimate we see that, choosing a bigger threshold $R$ if necessary, $\|\xi \|_{L^2}$ blows up only if $\|\xi \|_{H^{-\alpha}}$ does so.

\begin{lemma}\label{lem-3D-NSE-cut-off}
Let $\alpha\in(0,1/2)$ and $R>0$ be fixed. There exists a constant $C_{\alpha}$ such that the unique solution $\xi$ to \eqref{3D-NSE-cut-off} satisfies
  \begin{equation}\label{lem-3D-NSE-cut-off.1}
  \sup_{t\in[0,T]} \| \xi_t \|_{L^2}^2 + \int_0^T \|\nabla \xi_s\|_{L^2}^2\,\d s \leq \| \xi_0 \|_{L^2}^2 + C_\alpha (R+1)^{6/(1-2\alpha)}\, T \quad\P \mbox{-a.s.}
  \end{equation}
\end{lemma}

\begin{proof}
The proof is similar to that of Lemma \ref{lem-3D-NSE-limit}; here we give a sketch.

Using the Stratonovich structure and the divergence free property of the noise, we obtain the energy balance
  $$\d \|\xi \|_{L^2}^2 = 2\<\xi, -g_{\alpha,R}(\xi) \L_u\xi + \Delta\xi \>\,\d t = 2g_{\alpha,R}(\xi) \<\xi, \xi\cdot \nabla u\>\,\d t- 2 \|\nabla \xi\|_{L^2}^2\,\d t.$$
We estimate the nonlinear term as follows:
  $$\aligned
  g_{\alpha,R}(\xi) |\<\xi, \xi\cdot \nabla u\>| &\lesssim g_{\alpha,R}(\xi) \|\xi \|_{L^3}^3 \lesssim g_{\alpha,R}(\xi) \|\xi \|_{H^{1/2}}^3 \\
  &\leq g_{\alpha,R}(\xi) \|\nabla\xi \|_{L^2}^{3(1+2\alpha)/2(1+\alpha)} \|\xi \|_{H^{-\alpha}}^{3/2(1+\alpha)} \\
  &\leq \frac12 \|\nabla\xi \|_{L^2}^2 + C_\alpha g_{\alpha,R}(\xi) \|\xi \|_{H^{-\alpha}}^{6/(1-2\alpha)},
  \endaligned $$
where the last step follows from Young's inequality with exponents (which is allowed due to $\alpha\in (0,1/2)$)
  $$\frac{3(1+2\alpha)}{4(1+\alpha)} + \frac{1-2\alpha}{4(1+\alpha)}= 1. $$
The properties of $g_{\alpha,R}$ leads to
  $$\d \|\xi \|_{L^2}^2 + \|\nabla \xi\|_{L^2}^2\,\d t \leq C_\alpha (R+1)^{6/(1-2\alpha)} \,\d t. $$
This immediately implies the desired result.
\end{proof}

Note that the above estimate is independent of the noise parameters $\kappa>0$ and $\theta\in \ell^2(\Z^3_0)$.

\begin{remark}\label{rem-3D-blow-up}
If we consider equation \eqref{3D-NSE-limit} with cut-off $g_{\alpha,R}(\xi)$ in front of the nonlinear term, then similarly $\|\xi \|_{L^2}$ blows up only if $\|\xi \|_{H^{-\alpha}}$ does so. In particular, if $\kappa$ is chosen to be a function of $\xi_0$ as in Lemma \ref{lem-3D-NSE-limit} and $\| \xi_0 \|_{L^2}\leq R$, then the solutions to \eqref{3D-NSE-limit} coincide with those of the 3D Navier-Stokes equations with cut-off .
\end{remark}

Before moving forward, we present the following technical results which are classical.

\begin{lemma}\label{lem-3D-nonlinearity}
Let $V\in L^2(\T^3,\R^3)$ be a divergence free vector field.
\begin{itemize}
\item[\rm(i)] Let $\beta\in (0,1/2)$; if $V\in H^2$, then for any $f\in H^{-\beta}$, one has
  $$\|V\cdot \nabla f\|_{H^{-1-\beta}} \lesssim \|V \|_{H^2} \|f \|_{H^{-\beta}} .$$
\item[\rm(ii)] Let $\beta\in (0,1/2)$; if $f\in H^{2}$, then
  $$\|V\cdot \nabla f\|_{H^{-1-\beta}} \lesssim \|V \|_{H^{-\beta}} \|f \|_{H^{2}} .$$
\item[\rm(iii)] Let $\beta\in (0,3/2)$; if $V\in H^{s}$ for some $s\in [0,3/2-\beta]$ and $f\in H^{3/2-\beta-s}$, then
  $$\|V\cdot \nabla f\|_{H^{-1-\beta}} \lesssim \|V \|_{H^{s}} \|f \|_{H^{3/2-\beta-s}} .$$
\end{itemize}
\end{lemma}

\begin{proof}
(i) As $V$ is divergence free, we have $\|V\cdot \nabla f\|_{H^{-1-\beta}}= \|\nabla\cdot(V f)\|_{H^{-1-\beta}} \lesssim \|V f\|_{H^{-\beta}}$. By Sobolev embedding in the 3D case, $V\in C^s$ for any $s<1/2$. Fix $s\in (\beta,1/2)$; then for any $\varphi\in H^\beta$, it holds $\|V\varphi \|_{H^\beta} \lesssim \|V \|_{C^{s}}  \|\varphi \|_{H^\beta} \lesssim \|V \|_{H^2} \|\varphi \|_{H^\beta}$. Therefore,
  $$|\<Vf, \varphi\>| \le \|f\|_{H^{-\beta}} \|V\varphi \|_{H^\beta} \lesssim \|f\|_{H^{-\beta}} \|V \|_{H^2} \|\varphi \|_{H^\beta} $$
which implies $\|V f\|_{H^{-\beta}} \lesssim \|f\|_{H^{-\beta}} \|V \|_{H^2}$. Summarizing these estimates we obtain the first assertion.

(ii) The second estimate can be proved in the same way by using again that $\|V\cdot \nabla f\|_{H^{-1-\beta}} \lesssim \|V f\|_{H^{-\beta}}$ and exchanging the role of $V$ and $f$.

(iii) The proof of the last assertion is classical, see e.g. \cite[Lemma 2.1]{Temam95}; we give the proof for reader's convenience. For any $\varphi \in C^\infty$, by H\"older's inequality with $\frac1p+ \frac1q +\frac1r=1$,
  $$|\<V\cdot \nabla f, \varphi\>|= |\<f, V\cdot \nabla \varphi\>| \leq \|f \|_{L^p} \|V \|_{L^q} \|\nabla\varphi \|_{L^r}.$$
Recall the Sobolev embedding in 3D: for $\gamma\in [0,3/2)$, it holds $H^\gamma\subset L^a$ with $\frac1a = \frac12 - \frac{\gamma}3$; we can choose the exponents $p,q,r$ such that
  $$|\<V\cdot \nabla f, \varphi\>| \lesssim \|f\|_{H^{3/2-\beta-s}} \|V \|_{H^{s}} \|\nabla \varphi \|_{H^\beta} \lesssim \|f\|_{H^{3/2-\beta-s}} \|V \|_{H^{s}} \| \varphi \|_{H^{1+\beta}} .$$
Since $\varphi\in C^\infty$ is arbitrary, we obtain the desired result.
\end{proof}

Recall the nonlinearity with cut-off in \eqref{3D-NSE-cut-off}; for simplicity we define
  $$F(\xi)= g_{\alpha, R}(\xi) \L_u \xi, \quad \xi \in H^1, $$
where $g_{\alpha, R}(\xi)= g_R(\|\xi \|_{H^{-\alpha}})$ is the cut-off and $u=K\ast \xi$, $K$ being the Biot-Savart kernel on $\T^3$, which is a bounded linear operator from $H^a$ into $H^{a+1}$ for any $a\in \R$.

\begin{proposition}\label{prop-nonlin-3D}
Let $\alpha\in (0,1/2)$, it holds that
  $$\|F(\xi)- F(\tilde \xi)\|_{H^{-1-\alpha}} \lesssim \|\xi - \tilde\xi \|_{H^{-\alpha}} \big(\|\xi \|_{H^1} + \|\tilde\xi \|_{H^1} + \|\tilde \xi \|_{H^1} \|\tilde \xi \|_{L^2} \big). $$
\end{proposition}

\begin{proof}
Let $\tilde u= K\ast \tilde \xi$; then
  $$F(\xi)- F(\tilde \xi)=  g_{\alpha, R}(\xi) (\L_u \xi- \L_{\tilde u} \tilde \xi) + (g_{\alpha, R}(\xi) -g_{\alpha, R}(\tilde \xi))\L_{\tilde u} \tilde \xi =: I_1 + I_2. $$
We first estimate $I_2$. The Lipschitz property of $g_R$ yields
  $$|g_{\alpha, R}(\xi) -g_{\alpha, R}(\tilde \xi)| \leq \|g_R \|_{\rm Lip} \big| \|\xi\|_{H^{-\alpha}} -\|\tilde \xi\|_{H^{-\alpha}} \big| \leq \|\xi- \tilde\xi \|_{H^{-\alpha}}. $$
Note that both $\tilde u$ and $\tilde \xi$ are divergence free fields; by items (i)-(ii) of Lemma \ref{lem-3D-nonlinearity} with $\beta=\alpha$,
  $$\aligned
  \|\L_{\tilde u} \tilde \xi \|_{H^{-1-\alpha}} &\le \|\tilde u\cdot \nabla\tilde \xi\|_{H^{-1-\alpha}} +\|\tilde \xi\cdot \nabla \tilde u \|_{H^{-1-\alpha}} \lesssim \|\tilde u \|_{H^2} \|\tilde \xi \|_{H^{-\alpha}} \lesssim \|\tilde \xi \|_{H^1} \|\tilde \xi \|_{L^2}.
  \endaligned $$
Summarizing these estimates yields
  \begin{equation}\label{prop-nonlin-3D.1}
  \|I_2 \|_{H^{-1-\alpha}} \lesssim \|\xi- \tilde\xi \|_{H^{-\alpha}} \|\tilde\xi \|_{H^1} \|\tilde\xi \|_{L^2}.
  \end{equation}

It remains to estimate $I_1$. Since $0\leq g_R\leq 1$ we have
  $$\|I_1 \|_{H^{-1-\alpha}} \leq \|\L_u \xi- \L_{\tilde u} \tilde \xi \|_{H^{-1-\alpha}}\leq \|u\cdot\nabla \xi- \tilde u \cdot \nabla \tilde \xi \|_{H^{-1-\alpha}} + \|\xi\cdot\nabla u- \tilde \xi \cdot \nabla \tilde u \|_{H^{-1-\alpha}}. $$
Denote the two quantities on the right-hand side by $I_{1,1}$ and $I_{1,2}$. It holds that
  $$I_{1,1} \leq  \|(u-\tilde u)\cdot\nabla \xi \|_{H^{-1-\alpha}} + \|\tilde u \cdot \nabla(\xi- \tilde \xi) \|_{H^{-1-\alpha}}; $$
for the first term, applying Lemma \ref{lem-3D-nonlinearity}(iii) with $\beta=\alpha$ and $s=1-\alpha$ leads to
  $$\|(u-\tilde u)\cdot\nabla \xi \|_{H^{-1-\alpha}} \lesssim \|u-\tilde u\|_{H^{1-\alpha}} \|\xi \|_{H^{1/2}} \lesssim \|\xi- \tilde \xi \|_{H^{-\alpha}} \|\xi \|_{H^1}, $$
while the second term can be treated with Lemma \ref{lem-3D-nonlinearity}(i):
  $$\|\tilde u \cdot \nabla(\xi- \tilde \xi) \|_{H^{-1-\alpha}} \lesssim \|\tilde u\|_{H^2} \|\xi- \tilde \xi \|_{H^{-\alpha}} \lesssim \|\tilde \xi\|_{H^1} \|\xi- \tilde \xi \|_{H^{-\alpha}} .$$
To sum up, we arrive at
  $$\aligned
  I_{1,1} &\lesssim \big(\|\xi \|_{H^1} + \|\tilde \xi\|_{H^1}\big) \|\xi- \tilde \xi \|_{H^{-\alpha}}.
  \endaligned $$

Next, we turn to estimate the second quantity:
  $$\aligned
  I_{1,2} & \leq \| \xi \cdot \nabla(u- \tilde u) \|_{H^{-1-\alpha}} + \| (\xi-\tilde \xi)\cdot\nabla \tilde u \|_{H^{-1-\alpha}}.
  \endaligned $$
By Lemma \ref{lem-3D-nonlinearity}(iii) with $\beta=\alpha$ and $s=1/2$, we have
  $$\| \xi \cdot \nabla(u- \tilde u) \|_{H^{-1-\alpha}} \lesssim \|\xi\|_{H^{1/2}}\| u- \tilde u \|_{H^{1-\alpha}} \lesssim \|\xi\|_{H^1}\| \xi- \tilde \xi \|_{H^{-\alpha}}; $$
moveover, item (ii) of Lemma \ref{lem-3D-nonlinearity} implies that
  $$\| (\xi-\tilde \xi)\cdot\nabla \tilde u \|_{H^{-1-\alpha}} \lesssim \| \xi- \tilde \xi \|_{H^{-\alpha}} \|\tilde u\|_{H^2} \lesssim \| \xi- \tilde \xi \|_{H^{-\alpha}} \|\tilde \xi\|_{H^1}. $$
Therefore, we arrive at
  $$  I_{1,2} \lesssim \big(\|\xi \|_{H^1} + \|\tilde \xi\|_{H^1}\big) \|\xi- \tilde \xi \|_{H^{-\alpha}}. $$
Summarizing these arguments gives us
  $$\|I_1 \|_{H^{-1-\alpha}} \lesssim \big(\|\xi \|_{H^1} + \|\tilde \xi\|_{H^1}\big) \|\xi- \tilde \xi \|_{H^{-\alpha}}, $$
which, combined with \eqref{prop-nonlin-3D.1}, completes the proof.
\end{proof}

We will also need the following result whose proof is similar as above.

\begin{proposition}\label{prop-nonlin-3D-1}
It holds that
  $$\|F(\xi)- F(\tilde \xi)\|_{H^{-3/2}} \lesssim \|\xi - \tilde\xi \|_{L^2} \big(\|\xi \|_{L^2} + \|\tilde\xi \|_{L^2} + \|\tilde \xi \|_{L^2}^2 \big). $$
\end{proposition}

\begin{proof}
Let $I_1$ and $I_2$ be defined as in the proof of Proposition \ref{prop-nonlin-3D}. Again we start with estimating $I_2$. Clearly, we have
  $$|g_{\alpha, R}(\xi) -g_{\alpha, R}(\tilde \xi)| \leq \|\xi- \tilde\xi \|_{H^{-\alpha}}\leq \|\xi- \tilde\xi \|_{L^2}. $$
Recall that $\L_{\tilde u} \tilde \xi = \tilde u\cdot \nabla\tilde \xi -\tilde \xi\cdot \nabla \tilde u $; applying Lemma \ref{lem-3D-nonlinearity}(iii) with $\beta=1/2$ and $s=1$ leads to
  $$\aligned
  \|\tilde u\cdot \nabla\tilde \xi \|_{H^{-3/2}} \lesssim \|\tilde u \|_{H^1} \|\tilde \xi \|_{L^2}  \lesssim \|\tilde \xi \|_{L^2}^2;
  \endaligned $$
similarly, Lemma \ref{lem-3D-nonlinearity}(iii) with $\beta=1/2$ and $s=0$ yields  $\|\tilde \xi \cdot \nabla\tilde u\|_{H^{-3/2}} \lesssim \|\tilde \xi \|_{L^2} \|\tilde u \|_{H^1} \lesssim \|\tilde \xi \|_{L^2}^2 $. As a consequence,
  $$\|I_2 \|_{H^{-3/2}} \lesssim \|\xi- \tilde\xi \|_{L^2} \|\tilde \xi \|_{L^2}^2. $$

Next, we have
  $$\|I_1 \|_{H^{-3/2}} \leq \|u\cdot\nabla \xi- \tilde u \cdot \nabla \tilde \xi \|_{H^{-3/2}} + \|\xi\cdot\nabla u- \tilde \xi \cdot \nabla \tilde u \|_{H^{-3/2}}. $$
For the first quantity, by the triangle inequality and taking $\beta=1/2$ and $s=1$ in Lemma \ref{lem-3D-nonlinearity}(iii) leads to
  $$\aligned
  \|u\cdot\nabla \xi- \tilde u \cdot \nabla \tilde \xi \|_{H^{-3/2}} &\leq \|(u - \tilde u) \cdot \nabla \xi \|_{H^{-3/2}} + \| \tilde u \cdot \nabla (\xi- \tilde \xi) \|_{H^{-3/2}} \\
  & \lesssim \|u - \tilde u \|_{H^1} \|\xi \|_{L^2} + \| \tilde u \|_{H^1} \| \xi- \tilde \xi \|_{L^2} \\
  &\lesssim \big(\|\xi \|_{L^2} +\|\tilde \xi \|_{L^2} \big) \| \xi- \tilde \xi \|_{L^2}.
  \endaligned $$
In the same way, taking $s=0$ give us
  $$\aligned\|\xi\cdot\nabla u- \tilde \xi \cdot \nabla \tilde u \|_{H^{-3/2}}
  &\leq \| \xi \cdot \nabla (u - \tilde u)\|_{H^{-3/2}} + \| (\xi- \tilde \xi) \cdot \nabla \tilde u\|_{H^{-3/2}}\\
  &\lesssim \| \xi\|_{L^2} \| u - \tilde u\|_{H^1} + \| \xi-\tilde \xi\|_{L^2} \| \tilde u\|_{H^1} \\
  &\leq \big(\|\xi \|_{L^2} +\|\tilde \xi \|_{L^2} \big) \| \xi- \tilde \xi \|_{L^2} .
  \endaligned $$
Summing up the above estimates we complete the proof.
\end{proof}

\subsection{Proof of Theorem \ref{thm-3D-SNSE}} \label{subs-3D-NSE-proof}

The strategy of proof is similar to that of \cite[Theorem 1.6]{FlaLuo21}; there we only claimed the existence of some noise such that the blow-up probability of solutions to stochastic 3D Navier-Stokes equations \eqref{SNSE-vort-Ito-1} is sufficiently small. The quantitative estimates below allow us to choose noise parameters according to prescribed threshold for blow-up probability.

We begin with describing the idea of proof. We fix parameters $\alpha\in (0,1/2)$ and $L>0$. Recall the constant $C_1$ in Lemma \ref{lem-3D-NSE-limit} and fix $\kappa\geq C_1 L^4+1$ and $R=2L$. We shall compare the pathwise unique solution $\xi$ to
  \begin{equation}\label{proof-3D-SNSE}
  \aligned
  \d\xi &= \big[\Delta\xi - g_{\alpha,R}(\xi) \L_u\xi \big]\,\d t + \Pi(\circ\, \d W\cdot\nabla \xi) \\
  &=\big[\Delta\xi + S_{\theta^N}^{(3)}(\xi) - g_{\alpha,R}(\xi) \L_u\xi \big]\,\d t + \Pi(\d W\cdot\nabla \xi)
  \endaligned
  \end{equation}
with $\tilde\xi$ to
  \begin{equation}\label{proof-3D-NSE}
  \partial_t \tilde \xi = \Big(1+ \frac35 \kappa\Big) \Delta\tilde\xi - g_{\alpha,R}(\tilde\xi) \L_{\tilde u} \tilde\xi.
  \end{equation}
The noise $W$ is determined by $\kappa$ and $\theta^N\in \ell^2$ as in \eqref{noise}; the solution $\xi$ also depends on $N$ but we omit it for simplicity. Note that we add a cut-off to the nonlinearity of deterministic limit equation and denote its solution with a tilde. Both equations are equipped with the same initial data $\xi_0 \in H$ with $\|\xi_0 \|_{L^2} \leq L$. By Remark \ref{rem-3D-blow-up}, the solution $\tilde\xi$ coincides with the one to \eqref{3D-NSE-limit} without cut-off, which satisfies \eqref{lem-3D-NSE-limit.1}; in particular, $\|\tilde\xi_t \|_{L^2}$ decreases exponentially fast, and as a consequence, $\|\tilde\xi \|_{C([0,T]; H^{-\alpha})}$ is bounded by $\|\xi_0 \|_{L^2} \leq L$.

As mentioned above Lemma \ref{lem-3D-NSE-cut-off}, if $\sup_{t\in [0,T]} \|\xi_t \|_{H^{-\alpha}} \leq R=2L$, then $\xi$ also solves on $[0,T]$ the stochastic 3D Navier-Stokes equations \eqref{SNSE-vort-Ito-1} without cut-off. Furthermore, if, for some big $t_0\in [0,T]$, $\|\xi_{t_0} \|_{L^2}$ is small enough (e.g. less than the constant $r_0$ mentioned below \eqref{3D-NSE-energy-identity}), then we can regard $\xi_{t_0}$ as a initial condition and it is known that \eqref{SNSE-vort-Ito-1} admits a unique global solution for such initial data. Therefore, our purpose is to show that, with large probability, $\xi$ is close to $\tilde\xi$ both in $C([0,T]; H^{-\alpha})$ and in $L^2(T-1,T; L^2)$: the first closeness implies $\xi$ is a solution to \eqref{SNSE-vort-Ito-1} on $[0,T]$, while the second one will allow us to extend the solution to $+\infty$. Indeed, if for some large $T$ one has $\|\xi -\tilde\xi\|_{L^2(T-1,T; L^2)} \vee \| \tilde \xi \|_{L^2(T-1,T; L^2)} \leq r_0/2$, then there exists some $t_0\in [T-1, T]$ such that $\|\xi_{t_0}\|_{L^2} \leq r_0$. We shall make these considerations rigorous in the sequel. For simplicity, we assume $r_0\leq 1$.

Recall that $\tau(\xi_0,\kappa,\theta)$ is the blow-up time of solution to \eqref{SNSE-vort-Ito-1} (i.e. \eqref{SNSE-vort-Ito}) with initial data $\xi_0$. The above discussions imply that
  $$\{\tau(\xi_0,\kappa,\theta) =+\infty \} \supset \bigg\{ \sup_{t\in [0,T]} \|\xi_t\|_{H^{-\alpha}} \leq 2L \bigg\} \cap \big\{\|\xi \|_{L^2(T-1,T; L^2)}\leq r_0 \big\}; $$
as a result,
  \begin{equation}\label{blow-up-probab}
  \P(\tau(\xi_0,\kappa,\theta) <+\infty) \leq \P\bigg(\sup_{t\in [0,T]} \|\xi_t\|_{H^{-\alpha}}> 2L \bigg) + \P \big(\|\xi \|_{L^2(T-1,T; L^2)} > r_0 \big).
  \end{equation}
It remains to estimate the two probabilities on the right-hand side.

We make some preparations. Take $T= 1+ \log(2L/r_0)$, then by Lemma \ref{lem-3D-NSE-limit} it holds $\|\tilde\xi_t \|_{L^2} \leq r_0/2$ for all $t\geq T-1$. For parameters $L, \kappa, R$ and $ T$ chosen as above, Lemma \ref{lem-3D-NSE-cut-off} gives us
  \begin{equation}\label{solu-estimates.1}
  \sup_{t\in [0,T]} \|\xi_t\|_{L^2}^2 + \int_0^T \|\nabla \xi_s \|_{L^2}^2\,\d s \leq L^2 + C_\alpha (2L+1)^{6/(1-2\alpha)} \, T =: K_{\alpha, L, T} ;
  \end{equation}
in particular, by Poincar\'e's inequality,
  \begin{equation}\label{solu-estimates.2}
   \int_0^T \|\xi_s \|_{H^1}^2\,\d s = \int_0^T \big( \|\xi_s \|_{L^2}^2+ \|\nabla \xi_s \|_{L^2}^2 \big)\,\d s \leq \Big(\frac1{4\pi^2} + 1\Big)\int_0^T \|\nabla \xi_s \|_{L^2}^2\,\d s \leq 2K_{\alpha, L, T}.
  \end{equation}
Recalling the notation $F(\xi)$ defined above Proposition \ref{prop-nonlin-3D}; we rewrite the It\^o form of equation \eqref{proof-3D-SNSE} as follows:
  $$\d\xi = \Big(1+ \frac35\kappa \Big) \Delta\xi\,\d t + \Big(S_{\theta^N}^{(3)}(\xi)- \frac35\kappa \Delta\xi \Big)\,\d t - F(\xi)\,\d t + \Pi(\d W\cdot\nabla \xi). $$
By \eqref{thm-Ito-corrector.2} and \eqref{solu-estimates.2}, the second term on the right-hand side vanishes in $L^2(0,T; H^{-1-\alpha})$ as $N\to \infty$. Let $P_t= e^{t(1+3\kappa/5) \Delta},\, t\geq 0$; the above equation can be written in mild form as
  $$\xi_t= P_t \xi_0 + \int_0^t P_{t-s} \Big(S_{\theta^N}^{(3)}(\xi_s)- \frac35\kappa \Delta\xi_s \Big)\,\d s - \int_0^t P_{t-s} F(\xi_s) \,\d s +Z_t, $$
where the stochastic convolution is expressed as
  $$Z_t= \sqrt{\frac32 \kappa} \sum_{k,i} \theta_k \int_0^t P_{t-s} \Pi(\sigma_{k,i} \cdot\nabla \xi_s)\,\d W^{k,i}_s. $$
We also write the deterministic equation \eqref{proof-3D-NSE} in mild form:
  $$\tilde \xi_t= P_t \xi_0 - \int_0^t P_{t-s}F(\tilde\xi_s) \,\d s. $$
Therefore, defining the difference $\eta_t = \xi_t-\tilde\xi_t$, we have
  \begin{equation}\label{blow-up-difference}
  \eta_t = \int_0^t P_{t-s} \Big(S_{\theta^N}^{(3)}(\xi_s)- \frac35\kappa \Delta\xi_s \Big)\,\d s- \int_0^t P_{t-s} \big(F(\xi_s)- F(\tilde \xi_s) \big)\,\d s +Z_t.
  \end{equation}

\begin{lemma}\label{lem-blow-up-1}
There exists a constant $C=C_{\alpha, L, r_0}>0$ such that
  $$\P\bigg(\sup_{t\in [0,T]} \|\xi_t\|_{H^{-\alpha}}> 2L \bigg) \leq C_{\alpha, L, r_0} \bigg(\frac{\kappa}{N^{2\alpha}} + \frac{\kappa^{\alpha/2}}{N^\alpha} \bigg) . $$
\end{lemma}

\begin{proof}

Under the condition $\kappa\geq C_1 L^4+1$, the solution $\tilde\xi$ also solves the deterministic 3D Navier-Stokes equations \eqref{3D-NSE-limit} without cut-off and fulfills $\sup_{t\in [0,T]} \|\tilde\xi_t \|_{H^{-\alpha}} \leq \sup_{t\in [0,T]} \|\tilde \xi_t \|_{L^2} \leq L$. The triangle inequality implies
  $$\|\xi_t \|_{H^{-\alpha}} \leq \|\tilde\xi_t \|_{H^{-\alpha}}+ \|\eta_t \|_{H^{-\alpha}} \leq L +\|\eta_t \|_{H^{-\alpha}}. $$
Therefore,
  \begin{equation}\label{proof-probab}
  \P\bigg(\sup_{t\in [0,T]} \|\xi_t\|_{H^{-\alpha}}> 2L \bigg) \leq \P\bigg(\sup_{t\in [0,T]} \|\eta_t \|_{H^{-\alpha}} >L \bigg) \leq \frac1{L^2} \E\bigg[\sup_{t\in [0,T]} \|\eta_t \|_{H^{-\alpha}}^2 \bigg].
  \end{equation}
It suffices to estimate the right-hand side.

By \eqref{blow-up-difference} and Lemma \ref{lem:heat-kernel},
  $$\|\eta_t \|_{H^{-\alpha}}^2 \lesssim \frac1\kappa \int_0^t \Big\| S_{\theta^N}^{(3)}(\xi_s)- \frac35\kappa \Delta\xi_s \Big\|_{H^{-1-\alpha}}^2 \,\d s + \frac1\kappa \int_0^t \big\| F(\xi_s)- F(\tilde \xi_s) \big\|_{H^{-1-\alpha}}^2\,\d s + \|Z_t \|_{H^{-\alpha}}^2. $$
Now we apply \eqref{thm-Ito-corrector.2} and Proposition \ref{prop-nonlin-3D} to obtain
  $$\aligned
  \|\eta_t \|_{H^{-\alpha}}^2 \lesssim &\, \frac\kappa{N^{2\alpha}}\int_0^t \|\xi_s \|_{H^1}^2 \,\d s + \|Z_t \|_{H^{-\alpha}}^2 \\
  &\, + \frac1\kappa \int_0^t  \|\eta_s \|_{H^{-\alpha}}^2 \big(\|\xi_s \|_{H^1}^2 + \|\tilde\xi_s \|_{H^1}^2 + \|\tilde \xi_s \|_{H^1}^2 \|\tilde \xi_s \|_{L^2}^2 \big)\,\d s.
  \endaligned $$
Gronwall's inequality implies, for all $t\leq T$,
  $$\aligned
  \|\eta_t \|_{H^{-\alpha}}^2 \lesssim &\, \bigg(\frac\kappa{N^{2\alpha}} \|\xi \|_{L^2(0,T;H^1)}^2 + \sup_{t\in [0,T]} \|Z_t \|_{H^{-\alpha}}^2 \bigg)\\
  &\,\times \exp\bigg(\frac1\kappa \int_0^T \big(\|\xi_s \|_{H^1}^2 + \|\tilde\xi_s \|_{H^1}^2 + \|\tilde \xi_s \|_{H^1}^2 \|\tilde \xi_s \|_{L^2}^2 \big)\,\d s \bigg).
  \endaligned $$
Recall that $\tilde\xi$ solves the deterministic equation \eqref{3D-NSE-limit} and fulfills the estimates in \eqref{lem-3D-NSE-limit.1}; we apply Poincar\'e's inequality as in \eqref{solu-estimates.2} to get
  $$\int_0^T \big( \|\tilde\xi_s \|_{H^1}^2+ \|\tilde \xi_s \|_{H^1}^2 \|\tilde \xi_s \|_{L^2}^2 \big)\,\d s \leq 2L^2(L^2+1). $$
Combining the above computations with \eqref{solu-estimates.2}, we have, for any $t\leq T$,
  $$\|\eta_t \|_{H^{-\alpha}}^2 \lesssim \exp\bigg(\frac2\kappa \big[K_{\alpha, L,T} + L^2(L^2+1) \big] \bigg) \bigg(\frac{2\kappa}{N^{2\alpha}} K_{\alpha, L,T} + \sup_{t\in [0,T]} \|Z_t \|_{H^{-\alpha}}^2 \bigg). $$
As a result,
  $$\E\bigg[\sup_{t\in [0,T]} \|\eta_t \|_{H^{-\alpha}}^2 \bigg] \lesssim \exp\bigg(\frac{2}{\kappa} \big[K_{\alpha, L,T} + L^2(L^2+1) \big] \bigg) \bigg(\frac{\kappa}{N^{2\alpha}} K_{\alpha, L,T} + \E \bigg[\sup_{t\in [0,T]} \|Z_t \|_{H^{-\alpha}}^2 \bigg] \bigg). $$
Applying Theorem \ref{thm-stoch-convol-2} with $\delta=1+\frac35 \kappa$, $\beta=\alpha$, $\eps=\alpha/2$ and $d=3$ yields
  $$\E \bigg[\sup_{t\in [0,T]} \|Z_t \|_{H^{-\alpha}}^2 \bigg] \lesssim_{\alpha, T} \frac{\kappa^{\alpha/2}}{N^\alpha} K_{\alpha, L,T} ,  $$
where we have also used $\|\theta^N \|_{\ell^\infty} \sim N^{-3/2}$ and the bound \eqref{solu-estimates.1}. To sum up, we obtain
  \begin{equation}\label{proof-probab.1}
  \E\bigg[\sup_{t\in [0,T]} \|\eta_t \|_{H^{-\alpha}}^2 \bigg] \lesssim_{\alpha, T} \exp\bigg(\frac{2}{\kappa} \big[K_{\alpha, L,T} + L^2(L^2+1) \big] \bigg) \bigg(\frac{\kappa}{N^{2\alpha}} + \frac{\kappa^{\alpha/2}}{N^\alpha} \bigg) K_{\alpha, L,T}.
  \end{equation}
Inserting this estimate into \eqref{proof-probab} and noting that $T= 1+ \log(2L/r_0)$, we complete the proof.
\end{proof}

Next we estimate the second probability on the right-hand side of \eqref{blow-up-probab}.

\begin{lemma}\label{lem-blow-up-2}
Let $\kappa\geq (C_1 L^4+1)\vee (2(K_{\alpha, L,T} +1))^{1/2}$. For any $a\in (0,1)$ and $b\in (3/2, 7/2]$, there exists a constant $C=C_{\alpha, L, r_0, a, b}>0$ such that
  $$\P \big(\|\xi \|_{L^2(T-1,T; L^2)} > r_0 \big) \leq C_{\alpha, L, r_0,a, b} \bigg(\frac{\kappa^{1-\alpha}}{N^{2\alpha}} + \frac{1}{N^\alpha \kappa^{\alpha/2}} + \frac{C}{N^2} + \kappa^{\frac{4b-2ab -3a}{4(a+ b)}} N^{-\frac{2a}{a+b}} \bigg) . $$
\end{lemma}

\begin{proof}
Recall that $\|\tilde \xi_t \|_{L^2} \leq r_0/2$ for all $t\geq T-1$. The triangle inequality implies
  $$\|\xi \|_{L^2(T-1,T; L^2)} \leq \|\tilde \xi \|_{L^2(T-1,T; L^2)} + \|\eta \|_{L^2(T-1,T; L^2)} \leq \frac{r_0}2 + \|\eta \|_{L^2(T-1,T; L^2)}, $$
therefore,
  \begin{equation}\label{lem-blow-up-2.1}
  \P \big(\|\xi \|_{L^2(T-1,T; L^2)} > r_0 \big) \leq \P \big(\|\eta \|_{L^2(T-1,T; L^2)} > r_0/2 \big).
  \end{equation}
It remains to estimate the probability on the right-hand side.

We rewrite \eqref{blow-up-difference} as follows: for $t\in [T-1, T]$,
  $$\aligned
  \eta_t =&\, P_{t-T+1} (\eta_{T-1}) + \int_{T-1}^t P_{t-s} \Big(S_{\theta^N}^{(3)}(\xi_s)- \frac35\kappa \Delta\xi_s \Big)\,\d s \\
  &\, - \int_{T-1}^t P_{t-s} \big(F(\xi_s)- F(\tilde \xi_s) \big)\,\d s +Z_{T-1,t},
  \endaligned $$
where $Z_{T-1,t}$ is the stochastic convolution defined as before. We estimate the $L^2$-norm:
  \begin{equation}\label{lem-blow-up-2.2}
  \|\eta_t \|_{L^2}^2 \lesssim \| P_{t-T+1} (\eta_{T-1}) \|_{L^2}^2 + I_1(t) + I_2(t) + \|Z_{T-1,t} \|_{L^2}^2,
  \end{equation}
where
  $$\aligned
  I_1(t)&= \bigg\| \int_{T-1}^t P_{t-s} \Big(S_{\theta^N}^{(3)}(\xi_s)- \frac35\kappa \Delta\xi_s \Big)\,\d s \bigg\|_{L^2}^2, \\
  I_2(t) &= \bigg\| \int_{T-1}^t P_{t-s} \big(F(\xi_s)- F(\tilde \xi_s) \big)\,\d s \bigg\|_{L^2}^2.
  \endaligned $$

Recall that $P_t= e^{t(1+3\kappa/5) \Delta},\, t\geq 0$. First, for $\alpha\in (0,1/2)$ as above,
  $$\aligned
  \int_{T-1}^T \| P_{t-T+1}(\eta_{T-1}) \|_{L^2}^2 \,\d t &\lesssim_\alpha \int_{T-1}^T \frac1{(\kappa(t-T+1))^\alpha} \| \eta_{T-1} \|_{H^{-\alpha}}^2 \,\d t \lesssim_\alpha \frac1{\kappa^\alpha} \| \eta_{T-1} \|_{H^{-\alpha}}^2.
  \endaligned $$
Therefore, by \eqref{proof-probab.1},
  $$\aligned
  \E \int_{T-1}^T \| P_{t-T+1}(\eta_{T-1}) \|_{L^2}^2 \,\d t &\lesssim \frac1{\kappa^\alpha} \E \bigg[\sup_{t\in [0,T]}\| \eta_t \|_{H^{-\alpha}}^2 \bigg] \\
  &\lesssim \exp\bigg(\frac{2}{\kappa} \big[K_{\alpha, L,T} + L^2(L^2+1) \big] \bigg) \bigg(\frac{\kappa^{1-\alpha}}{N^{2\alpha}} + \frac{1}{N^\alpha \kappa^{\alpha/2}} \bigg) K_{\alpha, L,T}.
  \endaligned $$
Next, by the second estimate in Lemma \ref{lem:heat-kernel} and \eqref{thm-Ito-corrector.2},
  $$\aligned
  \int_{T-1}^T I_1(t)\,\d t &\lesssim \frac1{\kappa^2} \int_{T-1}^T \Big\| S_{\theta^N}^{(3)}(\xi_s)- \frac35\kappa \Delta\xi_s \Big\|_{H^{-2}}^2 \,\d s \\
  &\leq \frac{C}{N^2} \int_{T-1}^T \|\xi_s \|_{H^1}^2\,\d s \leq \frac{C}{N^2} K_{\alpha, L, T},
  \endaligned $$
where the last step follows from \eqref{solu-estimates.2}. Similarly, applying Proposition \ref{prop-nonlin-3D-1} yields
  $$\aligned
  \int_{T-1}^T I_2(t)\,\d t &\lesssim \frac1{\kappa^2} \int_{T-1}^T \big\| F(\xi_s)- F(\tilde \xi_s) \big\|_{H^{-2}}^2 \,\d s \\
  &\lesssim \frac1{\kappa^2} \int_{T-1}^T \|\eta_s \|_{L^2}^2 \big( \|\xi_s\|_{L^2}^2 + \|\tilde \xi_s\|_{L^2}^2 + \|\tilde \xi_s\|_{L^2}^4 \big) \,\d s .
  \endaligned $$
Recall that $\|\tilde \xi_s\|_{L^2} \leq r_0/2 \leq 1/2$ for $s\in [T-1,T]$; using the bound on $\xi$  we have
  $$\aligned
  \int_{T-1}^T I_2(t)\,\d t &\lesssim \frac1{\kappa^2} (K_{\alpha, L,T} + 1) \int_{T-1}^T \|\eta_s \|_{L^2}^2 \,\d s .
  \endaligned $$
Finally, with the bound \eqref{solu-estimates.1} in mind, we can repeat the proof of Theorem \ref{thm-stoch-convol} (with $\delta=1+ 3\kappa/5$ and $\nu=1$) to estimate the term involving stochastic convolution: for $0<a<1$ and $3/2< b\le 7/2$,
  $$\E \int_{T-1}^T \|Z_{T-1,t} \|_{L^2}^2 \,\d t \lesssim \kappa^{\frac{4b-2ab -3a}{4(a+ b)}} N^{-\frac{3a}{a+b}} K_{\alpha, L,T}. $$

Combining these estimates with \eqref{lem-blow-up-2.2} leads to
  $$\aligned
  \E \int_{T-1}^T \|\eta_t \|_{L^2}^2 \,\d t \lesssim&\, \exp\bigg(\frac{2}{\kappa} \big[K_{\alpha, L,T} + L^2(L^2+1) \big] \bigg) \bigg(\frac{\kappa^{1-\alpha}}{N^{2\alpha}} + \frac{1}{N^\alpha \kappa^{\alpha/2}} \bigg) K_{\alpha, L,T} + \frac{C}{N^2} K_{\alpha, L, T} \\
  &\, + \frac1{\kappa^2} (K_{\alpha, L,T} + 1)\, \E\! \int_{T-1}^T \|\eta_s \|_{L^2}^2 \,\d s + \kappa^{\frac{4b-2ab -3a}{4(a+ b)}} N^{-\frac{3a}{a+b}} K_{\alpha, L,T}.
  \endaligned $$
By the choice of $\kappa$, we have $\frac1{\kappa^2} (K_{\alpha, L,T} + 1) \leq \frac12$ and thus
  $$\aligned
  \E \int_{T-1}^T \|\eta_t \|_{L^2}^2 \,\d t & \lesssim \exp\bigg(\frac{2}{\kappa} \big[K_{\alpha, L,T} + L^2(L^2+1) \big] \bigg) \bigg(\frac{\kappa^{1-\alpha}}{N^{2\alpha}} + \frac{1}{N^\alpha \kappa^{\alpha/2}} \bigg) K_{\alpha, L,T} \\
  &\quad + \frac{C}{N^2} K_{\alpha, L, T}  + \kappa^{\frac{4b-2ab -3a}{4(a+ b)}} N^{-\frac{3a}{a+b}} K_{\alpha, L,T}\\
  & \lesssim \tilde K_{\alpha, L,T} \bigg(\frac{\kappa^{1-\alpha}}{N^{2\alpha}} + \frac{1}{N^\alpha \kappa^{\alpha/2}} + \frac{C}{N^2} + \kappa^{\frac{4b-2ab -3a}{4(a+ b)}} N^{-\frac{3a}{a+b}} \bigg).
  \endaligned $$
Therefore,
  $$\aligned
  \P \big(\|\eta \|_{L^2(T-1,T; L^2)} > r_0/2 \big) & \leq \frac4{r_0^2} \E \int_{T-1}^T \|\eta_t \|_{L^2}^2 \,\d t \\
  &\leq \frac4{r_0^2} \tilde K_{\alpha, L,T} \bigg(\frac{\kappa^{1-\alpha}}{N^{2\alpha}} + \frac{1}{N^\alpha \kappa^{\alpha/2}} + \frac{C}{N^2} + \kappa^{\frac{4b-2ab -3a}{4(a+ b)}} N^{-\frac{3a}{a+b}} \bigg).
  \endaligned $$
Substituting this estimate into \eqref{lem-blow-up-2.1} completes the proof.
\end{proof}

Finally we are ready to prove the main result.

\begin{proof}[Proof of Theorem \ref{thm-3D-SNSE}]
Taking $a=1/2$ and $b=2$ in Lemma \ref{lem-blow-up-2}, we deduce that for some constant $\hat C=\hat C_{\alpha, L, r_0}>0$ it holds
  $$\aligned
  \P \big(\|\xi \|_{L^2(T-1,T; L^2)} > r_0 \big)
  &\leq \hat C_{\alpha, L, r_0} \bigg(\frac{\kappa^{1-\alpha}}{N^{2\alpha}} + \frac{1}{N^\alpha \kappa^{\alpha/2}} + \frac{1}{N^2} + \frac{\kappa^{9/20}}{ N^{3/5}} \bigg) \\
  &\leq 2\hat C_{\alpha, L, r_0} \bigg(\frac{\kappa^{1-\alpha}}{N^{\alpha}} + \frac{\kappa^{9/20}}{ N^{3/5}} \bigg)
  \endaligned $$
since $\kappa \geq 1$. Combining this result with \eqref{blow-up-probab} and Lemma \ref{lem-blow-up-1}, we complete the proof.
\end{proof}

\appendix

\section{Proof of Theorem \ref{thm-Ito-corrector}} \label{sec-appendix}

This section is devoted to the proof of Theorem \ref{thm-Ito-corrector}; we shall only prove the first estimate since the second one can be proved in the same way by improving some estimates  in \cite[Section 5]{FlaLuo21}, see Remark \ref{rem-corrector-3D} for a brief discussion of necessary modifications.

Let $\Pi^\perp$ be the projection operator  orthogonal to $\Pi$: for any vector field $X\in L^2(\T^2, \R^2)$, $\Pi^\perp X$ is the gradient part of $X$. Formally,
  \begin{equation}\label{Leray-proj-1}
  \Pi^\perp X = \nabla \Delta^{-1} \div(X).
  \end{equation}
On the other hand, if $X= \sum_{l\in \Z^2_0} X_l e_l$, $X_l\in \mathbb C^2$, then
  \begin{equation}\label{Leray-proj-2}
  \Pi^\perp X= \sum_l \frac{l\cdot X_l}{|l|^2} l e_l = \nabla\bigg[ \frac1{2\pi {\rm i}} \sum_l \frac{l\cdot X_l}{|l|^2} e_l \bigg].
  \end{equation}

Now for a divergence vector field $v$ on $\T^2$, recall the Stratonovich-It\^o corrector
  $$S_\theta^{(2)} (v)= 2\kappa \sum_k \theta_k^2\, \Pi\big[ \sigma_k \cdot\nabla \Pi (\sigma_{-k} \cdot\nabla v) \big]. $$
Using the operator $\Pi^\perp$, we have
  \begin{equation}\label{decompositions}
  \Pi(\sigma_{-k}\cdot \nabla v) = \sigma_{-k}\cdot\nabla v- \Pi^\perp (\sigma_{-k}\cdot\nabla v).
  \end{equation}
Noting that, by \eqref{covariance-funct-1} and the fact that $\sigma_k \cdot \nabla \sigma_{-k} \equiv 0$,
  $$2\kappa \sum_k \theta_k^2\, \Pi\big[ \sigma_k \cdot\nabla (\sigma_{-k} \cdot\nabla v) \big] = \kappa\, \Pi [\Delta v]= \kappa\Delta v, $$
where the last step is due to the divergence free property of $v$, hence,
  $$S_\theta^{(2)}(v) = \kappa\Delta v - 2\kappa \sum_k \theta_k^2\, \Pi\big[ \sigma_k \cdot\nabla \Pi^\perp (\sigma_{-k} \cdot\nabla v) \big]. $$
We shall denote the second term on the right-hand side by $S_\theta^\perp (v)$. Therefore, the first assertion in Theorem \ref{thm-Ito-corrector} follows if we can prove
  \begin{equation}\label{alternative-limit}
  \bigg\| S_{\theta^N}^\perp (v)- \frac34 \kappa \Delta v \bigg\|_{H^{s-2-\alpha}} \leq C \frac{\kappa}{N^{\alpha}} \|v \|_{H^s} .
  \end{equation}

Now we assume the divergence free vector field $v$ has the Fourier expansion
  $$v= \sum_{l} v_l\, \sigma_l, $$
where the coefficients $\{v_l: l\in \Z^2_0 \} \subset \mathbb C$ satisfy $\overline{v_l}= v_{-l}$. We begin with finding the exact expression for $S_\theta^\perp (v)$.

\begin{lemma}\label{lem-extra-term}
We have
  \begin{equation}\label{lem-extra-term.1}
  S_\theta^\perp (v)= - 8\pi^2 \kappa \sum_l v_l \Pi\bigg\{ \bigg[ \sum_k \theta_k^2 (a_k \cdot l)^2 (a_l\cdot (k-l)) \frac{k-l}{|k-l|^2} \bigg] e_l \bigg\}.
  \end{equation}
\end{lemma}

\begin{proof}
We give a proof by using the formula \eqref{Leray-proj-2}; one can also proceed with \eqref{Leray-proj-1}, cf. \cite[Section 5]{FlaLuo21}. We have
  $$\nabla v(x)= \sum_l v_l \nabla \sigma_l(x) = 2\pi {\rm i} \sum_l v_l (a_l \otimes l) e_l(x). $$
Note that $\sigma_{-k}(x)= a_k e_{-k}(x)$; thus
  $$(\sigma_{-k}\cdot\nabla v)(x) = 2\pi {\rm i} \sum_l v_l (a_k \cdot l) a_l e_{l-k}(x). $$
By the first equality in \eqref{Leray-proj-2}, we have
  \begin{equation}\label{lem-extra-term.2}
  \aligned
  \Pi^\perp (\sigma_{-k}\cdot\nabla v)(x)&= 2\pi {\rm i} \sum_l v_l (a_k \cdot l) (a_l \cdot (l-k)) \frac{l-k}{|l-k|^2} e_{l-k}(x) .
  \endaligned
  \end{equation}
As a consequence,
  $$ \aligned
  &\, \big[ \sigma_k\cdot \nabla \Pi^\perp (\sigma_{-k,\alpha}\cdot\nabla v) \big](x)\\
  =&\, 2\pi {\rm i} \sum_l v_l (a_k \cdot l) (a_l \cdot (l-k)) \frac{l-k}{|l-k|^2} e_k(x) a_k \cdot\nabla e_{l-k}(x) \\
  =&\, (2\pi {\rm i})^2 \sum_l v_l (a_k \cdot l) (a_l \cdot (l-k)) \frac{l-k}{|l-k|^2} (a_k \cdot (l-k))  e_k(x) e_{l-k}(x) \\
  =& -4\pi^2 \sum_l v_l (a_k \cdot l)^2 (a_l\cdot (k-l)) \frac{k-l}{|k-l|^2} e_l(x),
  \endaligned $$
where in the last step we have used $a_k \cdot k=0$. This immediately gives us the desired identity.\end{proof}

In the next lemma we compute explicitly the projected quantity in \eqref{lem-extra-term.1}. Let $\angle_{k,l}$ be the angle between two vectors $k$ and $l$.

\begin{lemma}\label{lem-append-exress}
We have
  $$S_\theta^\perp (v)= - 8\pi^2 \kappa \sum_l v_l |l|^2 \bigg[ \sum_k \theta_k^2 \sin^2 (\angle_{k,l}) \frac{(a_l\cdot (k-l))^2}{|k-l|^2} \bigg] \sigma_l. $$
\end{lemma}

\begin{proof}
First, noting that $\{a_k, \frac{k}{|k|}\}$ is an ONB of $\R^2$ for any $k\in \Z^2_0$, we have
  $$(a_k \cdot l)^2= |l|^2 - \bigg(l\cdot \frac{k}{|k|} \bigg)^2 = |l|^2 \bigg[1- \bigg(\frac{l\cdot k}{|l|\, |k|} \bigg)^2\bigg] = |l|^2 \sin^2 \angle_{k,l}. $$
Therefore, we can rewrite \eqref{lem-extra-term.1} as
  \begin{equation}\label{lem-extra-term.3}
  \aligned
  S_\theta^\perp (v) &= - 8\pi^2 \kappa \sum_l v_l |l|^2 \Pi\bigg\{ \bigg[ \sum_k \theta_k^2 \sin^2 (\angle_{k,l}) (a_l\cdot (k-l)) \frac{k-l}{|k-l|^2} \bigg] e_l \bigg\} \\
  &= - 8\pi^2 \kappa \sum_l v_l |l|^2 \sum_k \theta_k^2 \sin^2 (\angle_{k,l}) \frac{a_l\cdot (k-l)}{|k-l|^2} \Pi((k-l) e_l),
  \endaligned
  \end{equation}
where the second step follows from the linearity of $\Pi$.

Since $\{\sigma_j\}_{j\in \Z^2_0}$ is a CONB in the space of square integrable and divergence free vector fields on $\T^2$ with zero mean, one has $\Pi((k-l) e_l) = \sum_j \<(k-l) e_l, \sigma_{-j} \> \sigma_j= ((k-l)\cdot a_l) \sigma_l$. Substituting this result into \eqref{lem-extra-term.3} leads to the desired equality.
\end{proof}

Recall the sequence $\theta^N \in \ell^2$ defined in \eqref{theta-N-def}. The next result is a crucial step for proving the limit \eqref{alternative-limit}.

\begin{proposition}\label{prop-2}
There exists some constant $C>0$ such that for any $l\in \Z^2_0$ and for all $N \geq 1$, it holds
  $$\bigg| \sum_{k} \big(\theta^N_k \big)^2 \sin^2(\angle_{k,l}) \frac{(a_l\cdot (k-l))^2}{|k-l|^2}- \frac38 \bigg| \leq C \frac{|l|}N. $$
\end{proposition}

Suppose we have already proved this result; we now turn to prove \eqref{alternative-limit}.

\begin{proof}[Proof of \eqref{alternative-limit}]
For any $N\geq 1$, by Lemma \ref{lem-append-exress},
  $$ S_{\theta^N}^\perp (v)= - 8\pi^2 \kappa \sum_l v_l |l|^2 \bigg[ \sum_{k} \big(\theta^N_k \big)^2 \sin^2(\angle_{k,l}) \frac{(a_l\cdot (k-l))^2}{|k-l|^2} \bigg] \sigma_l.$$
Since $v=\sum_l v_l \sigma_l$ one has
  $$\frac34 \kappa \Delta v= -3\pi^2\kappa \sum_l v_l |l|^2 \sigma_l, $$
therefore,
  $$\aligned
  &\, S_{\theta^N}^\perp (v) -\frac34 \kappa \Delta v = - 8\pi^2 \kappa \sum_l v_l |l|^2 \bigg[ \sum_{k} \big(\theta^N_k \big)^2 \sin^2(\angle_{k,l}) \frac{(a_l\cdot (k-l))^2}{|k-l|^2} - \frac38 \bigg] \sigma_l .
  \endaligned $$
Fix any $L>0$; we have
  $$\bigg\| S_{\theta^N}^\perp (v) -\frac34 \kappa \Delta v \bigg\|_{H^{s-2-\alpha}}^2 = K_{L,1} + K_{L,2},$$
where ($C=64\pi^4$)
  $$\aligned
  K_{L,1} &= C\kappa^2 \sum_{|l|\leq L} |v_l|^2 |l|^{2(s-\alpha)} \bigg| \sum_{k} \big(\theta^N_k \big)^2 \sin^2(\angle_{k,l}) \frac{(a_l\cdot (k-l))^2}{|k-l|^2} - \frac38 \bigg|^2, \\
  K_{L,2} &= C\kappa^2 \sum_{|l|> L} |v_l|^2 |l|^{2(s-\alpha)} \bigg| \sum_{k} \big(\theta^N_k \big)^2 \sin^2(\angle_{k,l}) \frac{(a_l\cdot (k-l))^2}{|k-l|^2} - \frac38 \bigg|^2 .
  \endaligned$$

Next we estimate the two quantities above. By Proposition \ref{prop-2},
  $$\aligned
  K_{L,1} &\leq C' \kappa^2 \sum_{|l|\leq L} |v_l|^2 \frac{|l|^{2(s-\alpha+1)}}{N^2} \leq C'\kappa^2 \frac{L^{2(1-\alpha)} }{N^2} \sum_{|l|\leq L} |v_l|^2 |l|^{2s} \leq C' \kappa^2 \frac{L^{2(1-\alpha)}}{N^2} \|v \|_{H^s}^2.
  \endaligned$$
Moreover, since $|a_l|=1$ and $\|\theta^N \|_{\ell^2}=1$, one has
  $$\aligned
  K_{L,2} &\leq C\kappa^2 \sum_{|l|> L} |v_l|^2|l|^{2(s-\alpha)} \bigg(\sum_{k} \big(\theta^N_k \big)^2  + \frac38 \bigg)^2\\
  &\leq 4C\kappa^2 \sum_{|l|> L} |v_l|^2 \frac{|l|^{2s}}{L^{2\alpha}} \leq C' \kappa^2 \frac{1}{L^{2\alpha}} \|v \|_{H^s}^2.
  \endaligned $$
Summarizing these estimates and taking $L=N$, we complete the proof of \eqref{alternative-limit}.
\end{proof}

Next we prove Proposition \ref{prop-2} for which we need some preparations.

\begin{lemma}\label{lem-differ}
For any $l\in \Z^2_0$ and all $N\geq 1$ it holds
  $$\bigg| \sum_{k} \big(\theta^N_k \big)^2 \sin^2(\angle_{k,l}) \bigg( \frac{(a_l\cdot (k-l))^2}{|k-l|^2} - \frac{(a_l\cdot k)^2}{|k|^2} \bigg) \bigg| \leq \frac{4|l|}N. $$
\end{lemma}

\begin{proof}
We have
  $$\bigg| \frac{(a_l\cdot (k-l))^2}{|k-l|^2} - \frac{(a_l\cdot k)^2}{|k|^2} \bigg| \leq \bigg\|\frac{(k-l)\otimes (k-l)}{|k-l|^2} - \frac{k\otimes k}{|k|^2} \bigg\| , $$
where $\|\cdot \|$ is the operator norm of matrices. Note that
  $$\frac{(k-l)\otimes (k-l)}{|k-l|^2} - \frac{k\otimes k}{|k|^2} = \frac{k-l}{|k-l|} \otimes \bigg(\frac{k-l}{|k-l|} - \frac{k}{|k|} \bigg) + \bigg(\frac{k-l}{|k-l|} - \frac{k}{|k|} \bigg) \otimes \frac{k}{|k|},$$
and thus
  $$\bigg\| \frac{(k-l)\otimes (k-l)}{|k-l|^2} - \frac{k\otimes k}{|k|^2} \bigg\| \leq 2\bigg| \frac{k-l}{|k-l|} - \frac{k}{|k|} \bigg|.$$
Next, since
  $$\frac{k-l}{|k-l|} - \frac{k}{|k|}= \bigg(\frac1{|k-l|} - \frac1{|k|}\bigg)(k-l) - \frac{l}{|k|}, $$
one has
  $$\bigg| \frac{k-l}{|k-l|} - \frac{k}{|k|} \bigg| \leq \frac{\big| |k| -|k-l| \big|}{ |k|} + \frac{|l|}{|k|} \leq 2 \frac{|l|}{|k|}. $$
Combining the above estimates we obtain
  $$\bigg| \frac{(a_l\cdot (k-l))^2}{|k-l|^2} - \frac{(a_l\cdot k)^2}{|k|^2} \bigg| \leq \bigg\|\frac{(k-l)\otimes (k-l)}{|k-l|^2} - \frac{k\otimes k}{|k|^2} \bigg\| \leq 4\frac{|l|}{|k|} . $$

Finally, recalling the definition of $\theta^N$ in \eqref{theta-N-def}, it holds
  $$\aligned &\, \sum_{k} \big(\theta^N_k \big)^2 \sin^2(\angle_{k,l}) \bigg| \frac{(a_l\cdot (k-l))^2}{|k-l|^2} - \frac{(a_l\cdot k)^2}{|k|^2} \bigg| \leq \sum_{N\leq |k|\leq 2N} \big(\theta^N_k \big)^2 \times 4\frac{|l|}{|k|} \leq \frac{4|l|}{N}.
  \endaligned $$
The proof is complete.
\end{proof}

\iffalse
Now we are ready to provide the

\begin{proof}[Proof of Proposition \ref{prop-2}]
By Lemma \ref{lem-differ}, it is sufficient to prove
  \begin{equation}\label{key-limit}
  \bigg| \frac{1}{\|\theta^N \|_{\ell^2}^2} \sum_{k} \big(\theta^N_k \big)^2 \sin^2(\angle_{k,l}) \frac{(a_l\cdot k)^2}{|k|^2} - \frac 38 \bigg| \leq C \frac{|l|}N .
  \end{equation}
\fi

We also need the following result.

\begin{lemma}\label{lem-proof}
Let $\{\theta^N \}_{N\geq 1} \subset \ell^2$ be given as in \eqref{theta-N-def}. There exists a $C>0$, independent of $N\geq 1$ and $l\in \Z^2_0$, such that
  \begin{equation*}
  \bigg| \sum_{k} \big(\theta^N_k \big)^2 \sin^2(\angle_{k,l}) \frac{(a_l\cdot k)^2}{|k|^2} - \frac{1}{\Lambda_N^2} \int_{\{N\leq |x|\leq 2N\}} \frac1{|x|^{2\gamma}} \sin^2(\angle_{x,l})  \frac{(a_l\cdot x)^2}{|x|^2} \,\d x \bigg| \leq \frac{C}N.
  \end{equation*}
\end{lemma}

\begin{remark}
In the 2D case, for any $l\in \Z^2_0$, $\{a_l, \frac{l}{|l|}\}$ is an ONB of $\R^2$; it is easy to see that $\frac{(a_l\cdot x)^2}{|x|^2} =\sin^2(\angle_{x,l}) $, thus the expression of integrand can be simplified. However, for the 3D case discussed in Remark \ref{rem-corrector-3D}, such simplification no longer holds; thus we decide to do the computations below using this more complicated expression, in order to shed some light on the computations in the 3D case.
\end{remark}

\begin{proof}[Proof of Lemma \ref{lem-proof}]
For any $l\in \Z^2_0$, we define the function
  $$g_l(x) = \sin^2(\angle_{x,l}) \frac{(a_l\cdot x)^2}{|x|^2}, \quad x\in \R^2,\, x\neq 0; $$
clearly, $\|g_l \|_\infty \leq 1$. We shall prove that there exists $C>0$, independent of $N\geq 1$ and $l\in \Z^2_0$, such that
  \begin{equation}\label{proof-lem.1}
  \aligned
  \bigg|\sum_{k} \big(\theta^N_k \big)^2 g_l(k) - \frac1{\Lambda_N^2}\int_{\{N\leq |x|\leq 2N\}} \frac{g_l(x)}{|x|^{2\gamma}}  \,\d x \bigg| \leq \frac{C}N .
  \endaligned
  \end{equation}
Let $\square(k)$ be the unit square centered at $k\in \Z^2$ such that all sides have length 1 and are parallel to the coordinate axes. Note that for all $k,l\in \Z^2$, $k\neq l$, the interiors of $\square(k)$ and $\square(l)$ are disjoint. Let $S_N = \bigcup_{N\leq |k|\leq 2N} \square(k)$; then,
  $$\bigg|\sum_{k} \big(\theta^N_k \big)^2 g_l(k) - \frac{1}{\Lambda_N^2} \int_{S_N} \frac{g_l(x)}{|x|^{2\gamma}} \,\d x\bigg| \leq \frac{1}{\Lambda_N^2} \sum_{N\leq |k|\leq 2N} \int_{\square(k)} \bigg|\frac{g_l(k)}{|k|^{2\gamma}} - \frac{g_l(x)}{|x|^{2\gamma}} \bigg|\, \d x.$$
For all $|k|\geq N\geq 1$ and $x\in \square(k)$, we have $|x-k|\leq \sqrt{2}/2$ and $|x|\geq 1/2$, thus
  $$\bigg|\frac{g_l(k)}{|k|^{2\gamma}} - \frac{g_l(x)}{|x|^{2\gamma}} \bigg| \leq \bigg|\frac{1}{|k|^{2\gamma}} - \frac{1}{|x|^{2\gamma}} \bigg| + \frac{|g_l(k) -g_l(x)|}{|x|^{2\gamma}} \leq C_1 \bigg(\frac{1}{|k|^{2\gamma+1}} + \frac{|g_l(k) -g_l(x)|}{|k|^{2\gamma}}\bigg) $$
for some constant $C_1$ depending on $\gamma>0$. Next,
  $$\aligned
  |g_l(k) -g_l(x)| &\leq |\sin^2(\angle_{k,l}) -\sin^2(\angle_{x,l})| + \bigg| \frac{(a_l\cdot k)^2}{|k|^2} - \frac{(a_l\cdot x)^2}{|x|^2} \bigg| \\
  &\leq 2|\sin(\angle_{k,l}) -\sin(\angle_{x,l})| + \bigg\| \frac{k\otimes k}{|k|^2} - \frac{x\otimes x}{|x|^2} \bigg\| \\
  &\leq 2| \angle_{k,l} -\angle_{x,l}| + 2 \bigg| \frac{k}{|k|} - \frac{x}{|x|}\bigg|.
  \endaligned $$
Since $x\in \square(k)$ and $|k|\geq N \geq 1$, we can find a constant $C_2>0$, independent of $l\in \Z^2_0$ and $N\geq 1$, such that
  $$|g_l(k) -g_l(x)| \leq \frac{C_2}{|k|}. $$
Summarizing the above discussions, we obtain
  $$\aligned
  \bigg|\sum_{k} \big(\theta^N_k \big)^2 g_l(k) - \frac{1}{\Lambda_N^2} \int_{S_N} \frac{g_l(x)}{|x|^{2\gamma}} \,\d x\bigg| &\leq \frac{1}{\Lambda_N^2}  \sum_{N\leq |k|\leq 2N} \int_{\square(k)} \frac{C_3}{|k|^{2\gamma+1}}\,\d x\\
  &\leq \frac {C_3}{N \Lambda_N^2} \sum_{N\leq |k|\leq 2N} \frac{1}{|k|^{2\gamma}} = \frac {C_3}N.
  \endaligned $$

Note that there is a small difference between the sets $\{N\leq |x|\leq 2N\}$ and $S_N$, but, in the same way, one can show that
  $$\bigg|\int_{\{N\leq |x|\leq 2N\}} \frac{g_l(x)}{|x|^{2\gamma}} \,\d x - \int_{S_N} \frac{g_l(x)}{|x|^{2\gamma}} \,\d x\bigg| \leq \frac CN \Lambda_N^2.$$
Indeed, for any $x\in \square(k)$ with $N\leq |k| \leq 2N$, one has $N-1 \leq |x| \leq 2N+1$. Therefore,
  $$S_N = \bigcup_{N\leq |k| \leq 2N} \square(k) \subset \{N-1 \leq |x| \leq 2N+1 \} =: T_N. $$
One also has
  $$R_N:= \{N+1 \leq |x| \leq 2N-1 \} \subset S_N.$$
Let $A\Delta B$ be the symmetric difference of subsets $A,B\subset \R^2$; then,
  $$\aligned
  &\, \bigg|\int_{\{N\leq |x|\leq 2N\}} \frac{g_l(x)}{|x|^{2\gamma}} \,\d x - \int_{S_N} \frac{g_l(x)}{|x|^{2\gamma}} \,\d x\bigg| \leq \int_{S_N\Delta \{N\leq |x|\leq 2N\}} \frac{g_l(x)}{|x|^{2\gamma}} \,\d x  \\
  \leq &\, \int_{S_N\Delta \{N\leq |x|\leq 2N\}} \frac{1}{|x|^{2\gamma}} \,\d x \leq \int_{T_N\setminus R_N} \frac{1}{|x|^{2\gamma}} \,\d x \leq \frac{C_4}{N^{2\gamma-1}} \leq \frac{C_5}{N} \Lambda_N^2,
  \endaligned $$
where the last step follows from
  $$\aligned
  \Lambda_N^2 &= \sum_{N\leq |k| \leq 2N} \frac1{|k|^{2\gamma}} \geq \frac1{(2N)^{2\gamma}}\, \#\{k\in \Z^2_0: N\leq |k| \leq 2N \} \geq \frac{C_6}{N^{2\gamma -2}}.
  \endaligned $$
It is easy to see that the constants $C_4, C_5$ and $C_6$ depend only on $\gamma$. The proof is complete.
\end{proof}

We are not ready to provide the

\begin{proof}[Proof of Proposition \ref{prop-2}]
In view of Lemma \ref{lem-proof}, we define
  $$ J_N = \frac{1}{\Lambda_N^2} \int_{\{N\leq |x|\leq 2N\}} \frac1{|x|^{2\gamma}} \sin^2(\angle_{x,l})  \frac{(a_l\cdot x)^2}{|x|^2} \,\d x .$$
To compute $J_N$, we consider the new coordinate system $(y_1, y_2)$ in which the coordinate axes are $a_l$ and $\frac{l}{|l|}$, respectively. Let $U$ be the orthogonal transformation matrix: $x=Uy$. For $i=1,2$, let ${\rm e}_i\in \R^2$ be such that ${\rm e}_{i,j}= \delta_{i,j}$, $1\leq j\leq 2$. We have
  $$a_l= U {\rm e}_1 \quad \mbox{and} \quad \frac{l}{|l|} = U {\rm e}_2.$$
Now $\angle_{x,l} = \angle_{Uy,U{\rm e}_2} = \angle_{y,{\rm e}_2}$ and
  \begin{equation}\label{proof.2}
  \aligned
  J_N &= \frac{1}{\Lambda_N^2} \int_{\{N\leq |y|\leq 2N\}} \frac1{|y|^{2\gamma}} \sin^2(\angle_{y,{\rm e}_2})\,  \frac{(U{\rm e}_1 \cdot Uy)^2}{|y|^2} \,\d y \\
  &= \frac{1}{\Lambda_N^2} \int_{\{N\leq |y|\leq 2N\}} \frac1{|y|^{2\gamma}} \sin^2(\angle_{y,{\rm e}_2})\, \frac{y_1^2}{|y|^2} \,\d y .
  \endaligned
  \end{equation}

We compute $J_N$ by changing the variables into the polar coordinate system:
  $$\begin{cases}
  y_1= r\cos \varphi, \\
  y_2 = r\sin \varphi, \\
 \end{cases} \quad N\leq r\leq 2N,\,0\leq \varphi< 2\pi .$$
In this system, $\varphi$ is the angle between $y$ and ${\rm e}_1$, hence $\sin^2(\angle_{y,{\rm e}_2}) = \cos^2 \varphi$. As a consequence,
  \begin{equation}\label{proof.4}
  \aligned
  J_N  &= \frac{1}{\Lambda_N^2} \int_N^{2N} \d r \int_0^{2\pi} \frac1{r^{2\gamma}} (\cos^4\! \varphi) \, r \, \d \varphi = \frac{3\pi}{4 \Lambda_N^2} \int_N^{2N}\frac{\d r}{r^{2\gamma -1}} .
  \endaligned
  \end{equation}
where in the second step we have used
  $$\aligned
  \int_0^{2\pi} \cos^4\! \varphi\ \d \varphi &= \frac14 \int_0^{2\pi} (1+ \cos 2\varphi)^2 \, \d \varphi = \frac14 \int_0^{2\pi} \bigg(1+ 2\cos 2\varphi + \frac{1+ \cos 4\varphi}2 \bigg) \, \d \varphi= \frac34 \pi .
  \endaligned $$
Following the proof of Lemma \ref{lem-proof} (it is much simpler here since the function $g$ can be taken identically 1), one can show
  \begin{equation}\label{proof.5}
  \bigg|\sum_{k} \big(\theta^N_k \big)^2 - \frac{1}{\Lambda_N^2} \int_{\{N\leq |x|\leq 2N\}} \frac{\d x}{|x|^{2\gamma}} \bigg|\leq \frac{C}{N}
  \end{equation}
for some constant $C>0$. Equivalently,
  $$\bigg| 1 - \frac{2\pi}{\Lambda_N^2} \int_N^{ 2N} \frac{\d r}{r^{2\gamma -1}} \bigg|\leq \frac{C}{N}. $$
Recalling \eqref{proof.4}, we obtain
  $$ \bigg|J_N - \frac 38 \bigg| \leq \frac{C'}N. $$
Combining this limit with Lemmas \ref{lem-differ} and \ref{lem-proof}, we complete the proof.
\end{proof}

Finally we discuss the necessary modifications for proving the estimate \eqref{thm-Ito-corrector.2} in the 3D case.

\begin{remark}\label{rem-corrector-3D}
Recall \cite[Corollary 5.3]{FlaLuo21} for the expression of $S_\theta^{(3),\perp}(v)$ (the part ``orthogonal'' to $S_\theta^{(3)}(v)$); note that there $\nu$ is the noise intensity, playing the role of $\kappa$ in this paper. Similarly to Lemma \ref{lem-append-exress} above, we can rewrite it as
  $$S_\theta^{(3),\perp}(v)= -6\pi^2 \kappa \sum_{l\in \Z^3_0} \sum_{i=1}^2 v_{l,i} |l|^2 \sum_{j=1}^2 \bigg[\sum_k \theta_k^2 \sin^2(\angle_{k,l}) \frac{(a_{l,i}\cdot (k-l)) (a_{l,j}\cdot (k-l))}{|k-l|^2}\bigg] \sigma_{l,j}. $$
Following Proposition \ref{prop-2}, we reformulate \cite[Proposition 5.4]{FlaLuo21} as below: there exists $C>0$ such that for all $l\in \Z^3_0$ and $N\geq 1$, it holds
  \begin{equation}\label{rem-corrector-3D.1}
  \bigg| \sum_k \theta_k^2 \sin^2(\angle_{k,l}) \frac{(a_{l,i}\cdot (k-l)) (a_{l,j}\cdot (k-l))}{|k-l|^2} - \frac4{15} \delta_{i,j} \bigg| \leq C \frac{|l|}N
  \end{equation}
for $i,j \in \{1,2\}$. With this result in hand, repeating the proof of \eqref{alternative-limit} yields
  $$\bigg\| S_{\theta^N}^{(3),\perp} (v)- \frac25 \kappa \Delta v \bigg\|_{H^{s-2-\alpha}} \leq C \frac{\kappa}{N^{\alpha}} \|v \|_{H^s} .$$

To show \eqref{rem-corrector-3D.1}, we simply replace \cite[Lemma 5.6]{FlaLuo21} by the following analogue of Lemma \ref{lem-proof}: there exists $C= C(\gamma)>0$, independent of $N\geq 1$ and $l\in \Z^3_0$, such that for any $i,j\in \{1,2 \}$,
  \begin{equation*}
  \bigg| \sum_{k} \big(\theta^N_k \big)^2 g_l^{i,j}(k) - \frac{1}{\Lambda_N^2} \int_{\{N\leq |x|\leq 2N\}} \frac{1}{|x|^{2\gamma}} g_l^{i,j}(x) \,\d x \bigg| \leq \frac{C}N,
  \end{equation*}
where $g_l^{i,j}(x)= \sin^2(\angle_{x,l}) \frac{(a_{l,i}\cdot x)(a_{l,j}\cdot x)}{|x|^2},\, x\in \R^3\setminus \{0\}$. Note that, by symmetry, the second quantity in the absolute value vanishes if $i\neq j$. The proof of this estimate is similar to that of Lemma \ref{lem-proof}, since $g_l^{i,j}$ enjoys the same properties as those of $g_l$.
\end{remark}

\section*{Declarations}

\noindent\textbf{Ethical Approval.} Not applicable. \smallskip

\noindent\textbf{Competing interests.} The author declares no competing interests. \smallskip

\noindent\textbf{Authors' contributions.} Not applicable. \smallskip

\noindent\textbf{Funding.} The author would like to thank the financial supports of the National Key R\&D Program of China (No. 2020YFA0712700), the National Natural Science Foundation of China (Nos. 11931004, 12090010, 12090014), and the Youth Innovation Promotion Association, CAS (Y2021002). \smallskip

\noindent\textbf{Availability of data and materials.} No new data and materials have been generated in the preparation of this paper.


\begin{thebibliography}{99} \setlength{\itemsep}{-1.5pt}

\bibitem{ACM19} G. Alberti, G. Crippa, A. L. Mazzucato. Exponential
self-similar mixing by incompressible flows. \emph{J. Amer. Math. Soc.} \textbf{32} (2019), 445--490.


%\bibitem{Alonson-Leon} D. Alonso-Or\'{a}n, A. B. de Le\'{o}n. On the Well-Posedness of Stochastic Boussinesq Equations with Transport Noise. \emph{J. Nonlinear Sci.} \textbf{30} (2020), no. 1, 175--224.

\bibitem{Arnold} L. Arnold. Stabilization by noise revisited. \emph{Z. Angew. Math. Mech.} 70 (1990), no. 7, 235--246.

\bibitem{ArnoldCW} L. Arnold, H. Crauel, V. Wihstutz. Stabilization of linear systems by noise. \textit{SIAM J. Control Optim.} \textbf{21} (1983), 451--461.

\bibitem{BBPS21} J. Bedrossian, A. Blumenthal, S. Punshon-Smith. Almost-sure enhanced dissipation and uniform-in-diffusivity exponential mixing for advection-diffusion by stochastic Navier-Stokes. \emph{Probab. Theory Related Fields} \textbf{179} (2021), no. 3--4, 777--834.

\bibitem{BedroBlum21} J. Bedrossian, A. Blumenthal, S. Punshon-Smith. Almost-sure exponential mixing of passive scalars by the stochastic Navier-Stokes equations. \emph{Ann. Probab.} \textbf{50} (2022), no. 1, 241--303.

\bibitem{BBPS18} J. Bedrossian, A. Blumenthal, and S. Punshon-Smith. Lagrangian chaos and scalar advection in stochastic fluid mechanics. \emph{J. Eur. Math. Soc. (JEMS)} \textbf{24} (2022), no. 6, 1893--1990.

\bibitem{BBPS19} J. Bedrossian, A. Blumenthal, and S. Punshon-Smith. The Batchelor spectrum of passive scalar turbulence in stochastic fluid mechanics at fixed reynolds number. \emph{Comm. Pure Appl. Math.} \textbf{75} (2022), no. 6, 1237--1291.

%\bibitem{BedrCoti} J. Bedrossian, M. Coti Zelati. Enhanced dissipation, hypoellipticity, and anomalous small noise inviscid limits in shear flows. \emph{Arch. Ration. Mech. Anal.} \textbf{224} (2017), no. 3, 1161--1204.

\bibitem{BGM19} J. Bedrossian, P. Germain, N. Masmoudi. Stability of the Couette flow at high Reynolds numbers in two dimensions and three dimensions. \emph{Bull. Amer. Math. Soc. (N.S.)} \textbf{56} (2019), no. 3, 373--414.

\bibitem{BedHe} J. Bedrossian, S. He. Suppression of blow-up in Patlak-Keller-Segel via shear flows. \emph{SIAM J. Math. Anal.} \textbf{49} (2017), no. 6, 4722--4766.

\bibitem{Berselli} L. C. Berselli, T. Iliescu, W. J. Layton. Mathematics of Large Eddy Simulation of Turbulent Flows, Springer, Berlin 2005.

\bibitem{BesFer17} H. Bessaih, B. Ferrario. The regularized 3D Boussinesq equations with fractional Laplacian and no diffusion. \emph{J. Differential Equations} \textbf{262} (2017), no. 3, 1822--1849.

\bibitem{BianFlan} L. A. Bianchi, F. Flandoli. Stochastic Navier-Stokes equations and related models. \emph{Milan J. Math.} \textbf{88} (2020), no. 1, 225--246.

\bibitem{Billingsley} P. Billingsley. Convergence of Probability Measures. Second edition. Wiley Series in Probability and Statistics: Probability and Statistics. A Wiley-Interscience Publication. \emph{John Wiley \& Sons, Inc., New York,} 1999.

\bibitem{BKL02} J. Bricmont, A. Kupiainen, R. Lefevere. Exponential mixing of the 2D stochastic Navier-Stokes dynamics. \emph{Comm. Math. Phys.} \textbf{230} (2002), no. 1, 87--132.

\bibitem{BCF} Z. Brze\'{z}niak, M. Capi\'{n}ski, F. Flandoli. Stochastic Navier-Stokes equations with multiplicative noise. \emph{Stochastic Anal. Appl.} \textbf{10} (1992), no. 5, 523--532.

\bibitem{BFM} Z. Brze\'{z}niak, F. Flandoli, M. Maurelli. Existence and uniqueness for stochastic 2D Euler flows with bounded vorticity. \emph{Arch. Ration. Mech. Anal.} \textbf{221} (2016), no. 1, 107--142.

\bibitem{BM19} Z. Brze\'{z}niak, M. Maurelli. Existence for stochastic 2D Euler equations with positive $H^{-1}$ vorticity. arXiv:1906.11523v2.

%\bibitem{ButMyt} O. Butkovski, L. Mytnik. Regularization by noise and flows of solutions for a stochastic heat equation. \textit{Ann. Probab.} \textbf{47} (2019), 165--212.

\bibitem{Constantin} P. Constantin, A. Kiselev, L. Ryzhik, A. Zlato\v{s}. Diffusion and mixing in fluid flow. \textit{Ann. of Math.} \textbf{168} (2008), 643--674.

\bibitem{CZDE} M. Coti Zelati, M.G. Delgadino, T.M. Elgindi. On the relation between enhanced dissipation timescales and mixing rates. \emph{Commun. Pure Appl. Math.} \textbf{73} (2020), no. 6, 1205--1244.

\bibitem{CZDFM} M. Coti Zelati, M. Dolce, Y. Feng, A. Mazzucato. Global existence for the two-dimensional Kuramoto-Sivashinsky equation with a shear flow. \emph{J. Evol. Equ.} \textbf{21} (2021), no. 4, 5079--5099.

%\bibitem{CZDri} M. Coti Zelati, T.D. Drivas. A stochastic approach to enhanced diffusion. arXiv:1911.09995.

\bibitem{CFH19} D. Crisan, F. Flandoli, D.D. Holm. Solution properties of a 3D stochastic Euler fluid equation. \emph{J. Nonlinear Sci.} \textbf{29} (2019), no. 3, 813--870.

\bibitem{DWZ21} W. Deng, J. Wu, P. Zhang. Stability of Couette flow for 2D Boussinesq system with vertical dissipation. \emph{J. Funct. Anal.} \textbf{281} (2021), no. 12, Paper No. 109255, 40 pp.

\bibitem{DKK04} D. Dolgopyat, V. Kaloshin, L. Koralov. Sample path properties of the stochastic flows. \emph{Ann. Probab.} \textbf{32} (2004), no. 1A, 1--27.

\bibitem{EMS01} W. E, J.C. Mattingly, Ya. Sinai. Gibbsian dynamics and ergodicity for the stochastically forced Navier-Stokes equation. Dedicated to Joel L. Lebowitz. \emph{Comm. Math. Phys.} \textbf{224} (2001), no. 1, 83--106.

\bibitem{Feff} C. L. Fefferman. Existence and smoothness of the Navier-Stokes equations, the millennium prize problems, Clay Math. Inst., Cambridge 2006, 57--67.

\bibitem{FengIyer} Y. Feng, G. Iyer. Dissipation enhancement by mixing. \emph{Nonlinearity} \textbf{32} (2019), no. 5, 1810--1851.

\bibitem{FFIT} Y. Feng, Y.-Y. Feng, G. Iyer, J.-L. Thiffeault. Phase separation in the advective Cahn-Hilliard equation. \emph{J. Nonlinear Sci.} \textbf{30} (2020), no. 6, 2821--2845.

\bibitem{FGL} F. Flandoli, L. Galeati, D. Luo. Scaling limit of stochastic 2D Euler equations with transport noises to the deterministic Navier--Stokes equations. \emph{J. Evol. Equ.} \textbf{21} (2021), no. 1, 567--600.

\bibitem{FGL21a} F. Flandoli, L. Galeati, D. Luo. Delayed blow-up by transport noise. \emph{Comm. Partial Differential Equations} \textbf{46} (2021), no. 9, 1757--1788.

\bibitem{FGL21c} F. Flandoli, L. Galeati, D. Luo. Eddy heat exchange at the boundary under white noise turbulence. \emph{Philos. Trans. Roy. Soc. A} \textbf{380} (2022), no. 2219, Paper No. 20210096, 13 pp.

\bibitem{FGL21b} F. Flandoli, L. Galeati, D. Luo. Quantitative convergence rates for scaling limit of SPDEs with transport noise. arXiv:2104.01740v2.

%\bibitem{FlaGat95} F. Flandoli, D. Gatarek. Martingale and stationary solutions for stochastic Navier-Stokes equations. \emph{Probab. Theory Related Fields} \textbf{102} (1995), 367--391.

\bibitem{FHLN21} F. Flandoli, M. Hofmanova, D. Luo, T. Nilssen. Global well-posedness of the 3D Navier--Stokes equations perturbed by a deterministic vector field. \emph{Ann. Appl. Probab.} \textbf{32} (2022), no. 4, 2568--2586.

\bibitem{FlaLuo21} F. Flandoli, D. Luo. High mode transport noise improves vorticity blow-up control in 3D Navier-Stokes equations. \emph{Probab. Theory Related Fields} \textbf{180} (2021), no. 1--2, 309--363.

\bibitem{FLL23} F. Flandoli, D. Luo, E. Luongo. 2D Smagorinsky type large eddy models as limits of stochastic PDEs, arXiv:2302.13614.

\bibitem{FlaLuongo22} F. Flandoli, E. Luongo. Heat diffusion in a channel under white noise modeling of turbulence. \emph{Math. Eng.} \textbf{4} (2022), no. 4, 1--21.

\bibitem{Flan21} F. Flandoli, E. Luongo. Stochastic Partial Differential Equations in Fluid Mechanics, Lecture Notes in Mathematics 2328, Springer, 2023.

\bibitem{FlaMas95} F. Flandoli, B. Maslowski. Ergodicity of the 2-D Navier-Stokes equation under random perturbations. \emph{Comm. Math. Phys.} \textbf{172} (1995), no. 1, 119--141.

\bibitem{FlaPap} F. Flandoli, U. Pappalettera. 2D Euler equations with Stratonovich transport noise as a large-scale stochastic model reduction. \emph{J. Nonlinear Sci.} \textbf{31} (2021), no. 1, Paper No. 24, 38 pp.

\bibitem{FlaPap21} F. Flandoli, U. Pappalettera. From additive to transport noise in 2D fluid dynamics. \emph{Stoch. Partial Differ. Equ. Anal. Comput.} \textbf{10} (2022), no. 3, 964--1004.

%\bibitem{FlaMah} F. Flandoli, A. Mahalov. Stochastic three-dimensional rotating Navier--Stokes equations: averaging, convergence and regularity. \emph{Arch. Rational Mech. Anal.} \textbf{205} (2012), no. 1, 195--237.

\bibitem{Galeati} L. Galeati. On the convergence of stochastic transport equations to a deterministic parabolic one. \emph{Stoch. Partial Differ. Equ. Anal. Comput.} \textbf{8} (2020), no. 4, 833--868.

\bibitem{GalGub} L. Galeati, M. Gubinelli. Mixing for generic rough shear flows. arXiv:2107.12115v1.

\bibitem{GalLuo}  L. Galeati, D. Luo. LDP and CLT for SPDEs with transport noise. \emph{Stoch. Partial Differ. Equ. Anal. Comput.} (2023). https://doi.org/10.1007/s40072-023-00292-y.

\bibitem{GessYar} B. Gess, I. Yaroslavtsev. Stabilization by transport noise and enhanced dissipation in the Kraichnan model, arXiv:2104.03949.

%\bibitem{GassiatGess} P. Gassiat, B. Gess. Regularization by noise for stochastic Hamilton-Jacobi equations. \emph{Probab. Theory Relat. Fields} \textbf{173} (2019), 1063--1098.

%\bibitem{Gess} B. Gess, Regularization and Well-Posedness by Noise for Ordinary and Partial Differential Equations. \emph{Stochastic partial differential equations and related fields}, 43--67, Springer Proc. Math. Stat., 229, \emph{Springer, Cham}, 2018.

%\bibitem{GessMau} B. Gess, M. Maurelli. Well-posedness by noise for scalar conservation laws. \emph{Comm. Partial Diff. Eq.} \textbf{43} (2018), no. 12, 1702--1736.

\bibitem{HM06} M. Hairer, J.C. Mattingly. Ergodicity of the 2D Navier-Stokes equations with degenerate stochastic forcing. \emph{Ann. of Math. (2)} \textbf{164} (2006), no. 3, 993--1032.

\bibitem{HLN19} M. Hofmanova, J. Leahy, T. Nilssen. On the Navier-Stokes equations perturbed by rough transport noise. \emph{J. Evol. Equ.} \textbf{19} (2019), 203--247.

\bibitem{Holm} D. D. Holm. Variational principles for stochastic fluid dynamics. \emph{Proc. Royal Soc.} A \textbf{471} (2015), 20140963.

\bibitem{Iyer} G. Iyer, X. Xu, A. Zlatos. Convection induced singularity suppression in the Keller-Siegel and other Non-liner PDEs. \emph{Trans. Amer. Math. Soc.} \textbf{374} (2021), no. 9, 6039--6058.

\bibitem{IJ20} A. D. Ionescu, H. Jia. Inviscid damping near the Couette flow in a channel. \emph{Comm. Math. Phys.} \textbf{374} (2020), no. 3, 2015--2096.

%\bibitem{Kuk-Shiri} S. Kuksin, A. Shirikyan. Mathematics of Two-Dimensional Turbulence. Cambridge Tracts in Mathematics, 194. \emph{Cambridge University Press, Cambridge}, 2012.

\bibitem{Kuk-Shiri01} S. Kuksin, A. Shirikyan. Ergodicity for the randomly forced 2D Navier-Stokes equations. \emph{Math. Phys. Anal. Geom.} \textbf{4} (2001), no. 2, 147--195.

\bibitem{CL19} O. Lang, D. Crisan. Well-posedness for a stochastic 2D Euler equation with transport noise. \emph{ Stoch. Partial Differ. Equ. Anal. Comput.} \textbf{11} (2023), no. 2, 433--480

\bibitem{Leray} J. Leray. Sur le mouvement d'un liquide visqueux emplissant l'espace. (French) \emph{Acta Math.} \textbf{63} (1934), no. 1, 193--248.

\bibitem{LTD11} Z. Lin, J.-L. Thiffeault, C. R. Doering. Optimal stirring strategies for passive scalar mixing. \emph{J. Fluid Mech.} \emph{675} (2011), 465--476.

\bibitem{Luo21} D. Luo. Convergence of stochastic 2D inviscid Boussinesq equations with transport noise to a deterministic viscous system. \emph{Nonlinearity} \textbf{34} (2021), 8311--8330.

\bibitem{Luo23} D. Luo. Regularization by transport noises for 3D MHD equations. \emph{Sci. China Math.} \textbf{66} (2023), no. 6, 1375--1394.

\bibitem{LuoSaal} D. Luo, M. Saal. A scaling limit for the stochastic mSQG equations with multiplicative transport noises. \emph{Stoch. Dynam.} \textbf{20} (2020), no. 6, 2040001, 21 pp.

\bibitem{LuoTang23} D. Luo, B. Tang. Stochastic inviscid Leray-$\alpha$ model with transport noise: Convergence rates and CLT. \emph{Nonlinear Anal.} \textbf{234} (2023), Paper No. 113301.

\bibitem{LuoWang23} D. Luo, D. Wang. Well posedness and limit theorems for a class of stochastic dyadic models. \emph{SIAM J. Math. Anal.} \textbf{55} (2023), no. 2, 1464--1498.

%\bibitem{Majda} A. Majda. Introduction to PDEs and waves for the atmosphere and ocean. Courant Lecture Notes in Mathematics, 9. \emph{New York University, Courant Institute of Mathematical Sciences, New York; American Mathematical Society, Providence, RI}, 2003.

\bibitem{MikRoz04} R. Mikulevicius, B. L. Rozovskii. Stochastic Navier-Stokes equations for turbulent flows. \emph{SIAM J. Math. Anal.} \textbf{35} (2004), no. 5, 1250--1310.

\bibitem{MikRoz05} R. Mikulevicius, B. L. Rozovskii. Global $L^2$-solutions of stochastic Navier-Stokes equations. \textit{Ann. Probab.} 33 (2005), no. 1, 137--176.

%\bibitem{Moffatt} H. K. Moffatt. Some remarks on topological fluid mechanics. \emph{An introduction to the geometry and topology of fluid flows (Cambridge, 2000)}, 3--10, NATO Sci. Ser. II Math. Phys. Chem., 47, \emph{Kluwer Acad. Publ., Dordrecht}, 2001.

%\bibitem{Paicu-Zhu} M. Paicu, N. Zhu. On the Yudovich's type solutions for the 2D Boussinesq system with thermal diffusivity. \emph{Discrete Contin. Dyn. Syst.} \textbf{40} (2020), 10, 5711--5728.

\bibitem{Pazy} A. Pazy, Semigroups of Linear Operators and Applications to Partial Differential Equations. Applied Mathematical Sciences, 44. \emph{Springer-Verlag, New York}, 1983.

\bibitem{PR07} C. Pr\'ev\^ot, M. R\"ockner. A concise course on stochastic partial differential equations. Lecture Notes in Mathematics, 1905. \emph{Springer, Berlin}, 2007.

%\bibitem{RozLot} B. L. Rozovsky, S. V. Lototsky. Stochastic evolution systems. Linear theory and applications to non-linear filtering. Second edition. Probability Theory and Stochastic Modelling, 89. \emph{Springer, Cham}, 2018.


\bibitem{Seis13} C. Seis. Maximal mixing by incompressible fluid flows. \emph{Nonlinearity} \textbf{26} (2013), 3279--3289.

\bibitem{Temam95} R. Temam, Navier-Stokes equations and nonlinear functional analysis. Second edition. CBMS-NSF Regional Conference Series in Applied Mathematics, 66. \emph{Society for Industrial and Applied Mathematics (SIAM), Philadelphia, PA}, 1995.

\bibitem{Wei21} D. Wei. Diffusion and mixing in fluid flow via the resolvent estimate. \emph{Sci. China Math.} \textbf{64} (2021), no. 3, 507--518.

\bibitem{WZZ20} D. Wei, Z. Zhang, W. Zhao. Linear inviscid damping and enhanced dissipation for the Kolmogorov flow. \emph{Adv. Math.} \textbf{362} (2020), 106963, 103 pp.

\bibitem{YaoZlatos} Y. Yao, A. Zlato\v{s}. Mixing and un-mixing by incompressible flows. \emph{J. Eur. Math. Soc. (JEMS)} \textbf{19} (2017), no. 7, 1911--1948.

\bibitem{Zlatos10} A. Zlato\v{s}, Diffusion in fluid flow: dissipation enhancement by flows in 2D. \emph{Commun. Partial Differ. Equ.} \textbf{35} (2010), no. 3, 496--534.


\end{thebibliography}
\end{document}